\documentclass[11pt,a4paper]{article}
\usepackage{amsmath,amssymb,amsthm,amsfonts,latexsym,graphicx,subfigure}
\usepackage[all,import]{xy}
\usepackage[numbers,sort&compress]{natbib}
\usepackage{indentfirst}
\usepackage{fancyhdr,enumerate}
\usepackage{bbm,color}
\usepackage{tikz}
\usetikzlibrary{shapes.geometric, arrows}
\usepackage{accents,cases}
\usepackage{multirow}
\usepackage[lined, ruled, linesnumbered, longend]{algorithm2e}

\usepackage{array,float}
\newcommand{\PreserveBackslash}[1]{\let\temp=\\#1\let\\=\temp}
\newcolumntype{C}[1]{>{\PreserveBackslash\centering}p{#1}}
\newcolumntype{R}[1]{>{\PreserveBackslash\raggedleft}p{#1}}
\newcolumntype{L}[1]{>{\PreserveBackslash\raggedright}p{#1}}

\newcommand{\RQ}{\mathcal{R}}

\makeatletter
\def\wbar{\accentset{{\cc@style\underline{\mskip8mu}}}}

\makeatother

\newcommand{\R}{\ensuremath{\mathbb{R}}}

\theoremstyle{plain}
\newtheorem{theorem}{Theorem}

\newtheorem{defn}{Definition}[section]

\newtheorem{lemma}{Lemma}
\newtheorem{remark}{Remark}[section]

\newtheorem{cor}{Corollary}[section]
\newtheorem{pro}{Proposition}[section]
\newtheorem{example}{Example}[section]
\newtheorem{Question}{Question}

\allowdisplaybreaks

\begin{document}

\title{New graph invariants based on $p$-Laplacian eigenvalues}
\author{Chuanyuan Ge\footnotemark[1]\and Shiping Liu\footnotemark[2] \and Dong Zhang\footnotemark[3]}

\footnotetext[1]{School of Mathematical Sciences,
University of Science and Technology of China, Hefei 230026, China.\\
Current address: School of Mathematics and Statistics,
Fuzhou University, Fuzhou 350108, China.\\
Email addresses: \texttt{gechuanyuan@mail.ustc.edu.cn},
\texttt{cyge@fzu.edu.cn}.}

\footnotetext[2]{School of Mathematical Sciences,
University of Science and Technology of China, Hefei 230026, China.\\
Email address: \texttt{spliu@ustc.edu.cn}.}

\footnotetext[3]{LMAM and School of Mathematical Sciences,
Peking University, Beijing 100871, China.\\
Email address: \texttt{dongzhang@math.pku.edu.cn}.}
\date{}\maketitle
\begin{abstract}
We present monotonicity inequalities for certain functions involving eigenvalues of $p$-Laplacians on signed graphs with respect to $p$. 
Inspired by such  monotonicity, we propose new spectrum-based graph invariants, called (variational) cut-off adjacency eigenvalues, 
that are relevant to certain eigenvector-dependent nonlinear eigenvalue problem. 
Using these invariants, we obtain new lower bounds for  the $p$-Laplacian variational eigenvalues, essentially giving the state-of-the-art spectral asymptotics for these eigenvalues. 
Moreover, based on such invariants, we establish two inertia bounds regarding the cardinalities of a maximum independent set and a minimum edge cover, respectively. The first inertia bound enhances the classical Cvetkovi\'c bound, and the second one implies that the $k$-th $p$-Laplacian variational eigenvalue is of the order $2^p$ as $p$ tends to infinity whenever $k$ is larger than the cardinality of a minimum edge cover of the underlying graph. 
We further discover an interesting connection between graph $p$-Laplacian eigenvalues and tensor eigenvalues and discuss applications of our invariants to spectral problems of tensors.   
\end{abstract}

\section{Introduction}

The study of the limit of $p$-Laplacian equations (e.g., asymptotic estimates of eigenvalues and  the limit of solutions) is a classical and important topic in the field of nonlinear PDEs.  
For instance, in 2005, Ishii and Loreti \cite{IL} analyzed the limit of solutions of the Dirichlet problem for the equation \begin{equation*}
  \left\{ \begin{aligned}   
  -\Delta_pu(x)=f(x) &\qquad\text{in }\Omega, \\
     u(x)=0 ~~~~~&\qquad \text{for }x\in \partial \Omega .
    \end{aligned}
    \right.
\end{equation*} as $p\to \infty$. 
However, in the graph setting, which has attracted a lot of attention in recent years, there are relatively few studies on the asymptotic behaviour of $p$-Laplacians. We will give a novel limit analysis for  discrete $p$-Laplacians on graphs. 
Let us start  with a brief introduction of the graph $p$-Laplacians.

As a classical nonlinear model equation on graphs, the graph $p$-Laplacians provide finer information about the network  topology. 
The spectrum of discrete $p$-Laplacian plays an important role in spectral graph theory, geometric group theory, machine learning, and image processing \cite{Amghibech,HeinBuhler2009,De,DM19,HeinBuhler2010,ZS06,ZHS06}. Unlike linear operators, the set of $p$-Laplacian eigenpairs is quite difficult to characterize. To understand this set, Hein-Tudisco \cite{TudiscoHein18} and Deidda-Putti-Tudisco \cite{DPT21} study the nodal domain count, Weyl-type inequalities and isoperimetric inequalities involving the graph $p$-Laplacian. 
Compared with the linear Laplacian, these inequalities provide better estimates for Cheeger-type constants on graphs.  
 Just as an example, in \cite{Amghibech}, the Cheeger inequality on graphs reads
 \[
 \frac{2^{p-1}}{p^p}h^p\le \lambda_2^{(p)}\le 2^{p-1}h,
 \]
 which refines the classical Cheeger inequality $h^2/2\le \lambda^{(2)}_2\le 2h$ when $1\le p<2$. It is easy to see that the lower bound $\frac{2^{p-1}}{p^p}h^p$ converges to $0$ very quickly when $p$ tends to $+\infty$, but  $\lambda_2^{(p)}$ can go to $+\infty$ when, for example, the graph is complete. 
 That is, from the point of view of asymptotic analysis, this lower bound provided by the Cheeger inequality is not good. Indeed, the lower bound provided by the multi-way Cheeger inequality \cite{GLZ22+,TudiscoHein18} for the $k$-th variational eigenvalue has similar issues. 
Surprisingly, there is no better known lower bound for general variational  eigenvalues of graph $p$-Laplacians in the literature. Only certain lower bounds under various restrictions about the smallest nonzero and largest eigenvalues are obtained in \cite{Amghibech06,De,DM19}. 
In addition, many other general questions on estimating the eigenvalues of graph $p$-Laplacian are still far from being understood.

Recently, the study on signed graphs has received increasing attention \cite{GL21+,Zaslavsky10} and, in particular, $p$-Laplacian spectral theory has been tentatively established on signed graphs \cite{GLZ22+}.  A signed graph $\Gamma=(G,\sigma)$ is a graph $G=(V,E)$ equipped with a signature $\sigma:E\to\{+1,-1\}$ on its edge set $E$. There are significant differences between the spectra of signed graphs and those of ordinary graphs. Some fundamental results on $p$-Laplacian eigenpairs of ordinary graphs no longer hold for signed graphs. However it is worth noting that we can use signed graphs as a tool for studying the linear and nonlinear graph eigenvalue problems \cite{GL21+,GLZ22+}. 

Along this spirit, in this paper we provide new insights into the $p$-Laplacian spectral behaviour by introducing  new invariants for signed graphs. For simplicity, throughout the introduction, all graphs and signed graphs are assumed to be unweighted with an identically zero potential function. However, all the results established in this paper remain valid in the weighted setting with suitable potential functions. Our new invariants are actually spectrum-based quantities on signed graphs, which are defined as a vector $(L_1,L_2,\ldots, L_n)$ consisting of the limits of the variational eigenvalues divided by $2^{p}$ as $p$ tends towards $+\infty$ (see Definition \ref{def:cutoffadj}). 
Our invariants follow from a new monotonicity inequality of $p$-Laplacian eigenvalues for varying $p$ (see Theorem \ref{thm:Monotonicity}), which covers recent monotonicity results in 
\cite[Proposition 3.1]{HW20} and \cite[Theorem 1.1]{Zhang}. 
More importantly, the new spectrum-based invariants can be characterized variationally (see Theorem \ref{thm:varitional}), and they are associated with the eigenvalues of the adjacency matrices of certain subgraphs via  non-trivial inequalities (see Theorem \ref{thm:L_n's antibalanced}). 
This helps us to derive new  growth estimates of the variational eigenvalues of graph $p$-Laplacians, which greatly strengthens the existing  estimates provided by the multi-way Cheeger inequalities. For a signed graph $\Gamma$, let $\lambda_k(A_{\Gamma})$ denote the $k$-th eigenvalue of its adjacency matrix $A_{\Gamma}$. For the $p$-Laplacian on an $n$-vertex signed graph $\Gamma$ with an identically zero potential function, our new lower bounds of its variational eigenvalues read as
\begin{equation}\label{eq:1}
\lambda_k^{(p)}
\ge 2^{p-1}\lambda_k(-A_{\Gamma'}),\;\; k=1,\cdots,n,
\end{equation}
where $\Gamma'$ is any spanning signed subgraph of $\Gamma$, i.e., any signed subgraph of $\Gamma$ on the same vertex set as $\Gamma$. In addition, in the spirit of \cite{IL}, we  analyze  the limit of the eigenfunctions corresponding to the largest eigenvalues of graph $p$-Laplacians on anti-balanced graphs. It is important to note that this analysis is only applicable in the graph setting, since $p$-Laplacians on Euclidean domains do not have ``largest" eigenvalues. Also, the new invariants are indispensable for analyzing these limits. 
We refer to  Theorem \ref{thm:lower bound} and  Section \ref{Lower bound} for more precise and strengthened results on lower bounds for variational eigenvalues of graph $p$-Laplacians.

Another progress is made in the study of hidden relations to the cardinalities of a maximum independent set and a minimum edge cover. In fact, the proposed invariants also lead to a new perspective on inertia bounds for the cardinality of a maximum independent set \cite{Cvetkovic} (see also Theorem \ref{thm:Inertia bound}).
A signed version of the inertia bound was essentially used in the proof of the Sensitivity Conjecture \cite{Huang19}. 
Our result provides a refined analog for the signed inertia bound in a general setting. Precisely, we prove that for the cardinality $\alpha$ of a maximum independent set  and the cardinality $\beta$ of a minimum edge cover of a graph, we have 
\begin{equation}\label{eq:inertia-L}
 \alpha\le \min_\sigma\#\{k:L_k=0\}, 
\end{equation}
\begin{equation}\label{eq:inertia-L_2}
 \beta\ge \max_\sigma\#\{k:L_k=0\}, 
\end{equation}
where $L_k$ is the $k$-th invariant with respect to the signature $\sigma$. Moreover, the inertia bound \eqref{eq:inertia-L} covers the Cvetkovi$\acute{c}$ bound, see Theorem \ref{thm:Inertia bound} and Proposition \ref{pro:cover}. One advantage of our inertia bounds is that they do not depend on the choices of edge weights and vertex measures. Indeed, let $\Gamma=(G,\sigma)$ be a signed graph with an identically zero potential function. Let $\{L_k\}_{k=1}^n$ and $\{L_k'\}_{k=1}^n$ denote the cut-off adjacency eigenvalues of the $p$-Laplacian of $\Gamma$ corresponding to two different choices of edge weights and vertex measures, respectively. It follows from the proof of Proposition \ref{pro:weight and measure} that there exist constants $C_0,C_1>0$ such that $C_0L_k\leq L'_k\leq C_1L_k$ for every $k=1,\ldots,n$. This indicates that, for varying edge weights and vertex measures, the quantities on the right hand side of \eqref{eq:inertia-L} and \eqref{eq:inertia-L_2} stay put. However, in the Cvetkovi$\acute{c}$ bound, one needs to optimize over all the choices of edge weights. 

With a bound regarding the cardinality of an edge cover, we further obtain that $\lambda^{(p)}_k\sim 2^{p}$ as $p$ tends to infinity, whenever $k$ is larger than the cardinality of an edge cover. Precisely, we have for any $k$ bigger than the cardinality of an edge cover and any $p>1$ that 
$$ c\,2^{p}\le \lambda_k^{(p)}\le C\,2^{p}, $$
 where the positive constants $c$ and $C$ are 
independent of $p$. 
In fact, $\lambda_k^{(p)}=\Theta(2^p)$ is the fastest possible growing rate for $\lambda_k^{(p)}$ as $p\to+\infty$. On the other hand, we show that $\lambda_k^{(p)}$ is uniformly bounded whenever $k$ is no larger than the cardinality of a maximal independent set.

We find an interesting connection between the spectral problem of signed graph $p$-Laplacians with that of the tensors. The spectrum of a (symmetric) tensor is a widely studied object and has attracted much attention due to its wide range of applications \cite{HL13,Qi05,QL17book,QCC18book,Lim15,Qi07}.  However, bounds for the eigenvalues are not well understood, as the nonlinearity of a tensor increases the difficulty of studying these eigenvalues. For instance, there are few known lower bounds for tensor eigenvalues, and most are lower bounds for the largest one  \cite{NQZ09}. 
Moreover, to the best of our knowledge, no one has found it possible to give lower bound estimates for the general eigenvalues of special tensors using linear spectra. We show that, for even $p$, the signed graph $p$-Laplacian eigenvalue problem is equivalent to a tensor eigenvalue problem (see Proposition \ref{prop:tensor}). Based on this equivalence, we establish the first such kind of lower bound estimates for the eigenvalues of the particular tensor induced by the $p$-Laplacian when $p$ is even (see Theorem \ref{thm:tensor}). On the other hand, the equivalence allows the existing theory and algorithms for tensor eigenvalues to be applied directly to the signed graph $p$-Laplacian setting.
Along these lines, the first author and Qin \cite[Example~3.4]{GQ25} employed the tensor eigenvalue algorithm developed by Cui, Dai, and Nie \cite{CDN14} to determine all the eigenvalues of the \(4\)-Laplacian on a signed cycle \(\Gamma=(C_4,\sigma)\), where \(C_4\) denotes the cycle graph on four vertices and \(\sigma\) is a signature. In general, it remains an open problem whether the graph \(p\)-Laplacian has only finitely many eigenvalues for \(p\in(1,\infty)\setminus\{2\}\). As observed in \cite[Theorem~3.3]{GQ25}, combining the above-mentioned connection with known properties of tensor eigenvalues shows that the signed graph \(p\)-Laplacian has only finitely many eigenvalues whenever \(p\) is an even integer.

A variational characterization shows that our new graph invariants $(L_1,\ldots, L_n)$ can be written as the min-max formulation of a nonlinear functional which is closely related to the graph adjacency matrix. More precisely, we have for any $1\leq k\leq n$ that, 
\begin{equation}\label{eq:variational}
L_k:=\lim_{p\to \infty}2^{-p}\lambda_k^{(p)}=\min_{B\in \mathcal{F}_k(\mathcal{S}_2)}\max_{g\in B}\sum_{\{i, j\}\in E}  \max\left\{-g(i)\sigma_{ij}g(j), 0\right\},
\end{equation}
where $\mathcal{S}_2$ is the set of functions with unit norm, and $\mathcal{F}_k(\mathcal{S}_2)$ is the set consisting of closed symmetric subsets of $\mathcal{S}_2$ with index at least $k$. Here, we denote by  $\sigma_{ij}$ the signature of each edge $\{i,j\}\in E$.
Due to this expression, we call the proposed graph invariants the (variational) cut-off adjacency eigenvalues. 
We also obtain an interlacing theorem for the cut-off adjacency eigenvalues under particular manipulations of the underlying graphs. 
It should be noted that most of the proofs for the results in this paper are building upon the variational characterization (\ref{eq:variational}).

The paper is structured as follows. We begin with a careful review of existing terms, notations, and necessary related results in Section \ref{sec:pre}. In Section \ref{sec:mono},  we prove monotonicity properties of the variational eigenvalues of graph  $p$-Laplacians with respect to $p$, which is the starting point for our new invariants.  
In Section \ref{graph invariant}, we address the main purpose of this work by introducing the new spectrum-based invariants, and prove  lower bounds for variational eigenvalues of graph  $p$-Laplacians based on these invariants in Section  \ref{Lower bound}. 
Two inertia bounds regarding the cardinalities of a maximum independent set and a minimum edge cover are given in Section \ref{sec:inertia}, and an application to  tensor eigenvalues is presented in Section \ref{Tensor}. In Section \ref{sec:exam}, we give some explicit examples. Finally, in Section \ref{sec:outlook}, we conclude with remarks on associated eigenfunctions and discuss directions for future research.

\section{Preliminaries}
\label{sec:pre}
In this section, we introduce our settings and notations. Let $G=(V,E)$ be a finite graph where $V=\{1,2,\ldots,n\}$ is a vertex set and $E\subseteq V\times V$ is an edge set. We use the notation $i\sim j$ to indicate $\{i,j\}\in E$. We only consider the undirected graph without self-loops in this paper. We always use $n$ to denote $|V|$ if it is clear. Let $C(V)$  denote all real functions on $V$. We freely interchange  between the function $f\in C(V)$ and the vector $(f(1),f(2),\ldots,f(n))\in \mathbb{R}^n$. 

A signed graph $\Gamma=(G,\sigma)$ is a graph $G=(V,E)$ with a signature $\sigma:E \to \{+1,-1\}$.  For convenience, we use $\sigma_{ij}$ to denote $\sigma(\{i,j\})$. 
For a subgraph $G'=(V',E')$ of $G=(V,E)$, i.e., $V'\subset V$ and $E'\subset E$, we write $(G',\sigma)$ to denote the signed graph of $G'$ whose signature is the restriction of $\sigma$ on $E'$. We call $\Gamma':=(G',\sigma)$ a \emph{subgraph} of $\Gamma$ and $-\Gamma:=(G,-\sigma)$ the \emph{negation} of the signed graph $\Gamma$.

\begin{defn}[Switching]
    Let $\Gamma=(G,\sigma)$ be a signed graph with $G=(V,E)$. A function $\tau$ is called a switching function if it maps from $V$ to $\{+1,-1\}$. We define the operation of switching the signed graph $\Gamma=(G,\sigma)$ by $\tau$ to be changing the signature $\sigma$ to $\sigma^{\tau}$ where $$\sigma^{\tau}_{ij}:=\tau(i)\sigma_{ij}\tau(j),$$
for any $\{i,j\}\in E.$ We denote by $\Gamma^{\tau}$ the signed graph obtained via switching $\Gamma$ by $\tau$.
\end{defn}
\begin{defn}
    We say two signed graphs $\Gamma=(G,\sigma)$ and $\tilde{\Gamma}=(G,\tilde{\sigma})$ are switching equivalent if there exists a switching function $\tau$ such that $\sigma^\tau=\tilde{\sigma}$.
\end{defn}
Next, we give the definitions of balanced and antibalanced graphs. The following definition is equivalent to the original one by Harary \cite{Harary53} due to the Zaslavsky’s switching lemma \cite{Zaslavsky82}.
\begin{defn}
    A signed graph $\Gamma=(G,\sigma)$ is balanced (resp., antibalanced) if it is switching equivalent to the signed graph $\Gamma^+=(G,\sigma^+)$ (resp., $\Gamma^-=(G,\sigma^{-})$) where $\sigma^+\equiv+1$ (resp., $\sigma^-\equiv-1$).
\end{defn}
In this paper, we focus on $p$-Laplacians  on a signed graph $\Gamma=(G,\sigma)$ for $p>1$. Define $\Psi_p:\mathbb{R}\to \mathbb{R}$  as follows: 
 \begin{equation*}
  \Psi_p(t)=\left\{ \begin{aligned}
        &|t|^{p-2}t\quad &&t\neq 0,\\
        &0 \quad && t=0.
    \end{aligned}
    \right.
\end{equation*}
Given  an edge weight $w:E\to \R^+$ and a potential function  $\kappa:V\to\R$, the signed $p$-Laplacian $\Delta_p^{\sigma}: C(V)\to C(V)$ is defined as follows:
$$\Delta_p^\sigma f(i):=\sum_{j\sim i}w_{ij}\Psi_p\left(f(i)-\sigma_{ij}f(j)\right)+\kappa_i\Psi_p(f(i)),$$
for any $f\in C(V)$ and $i\in V$, where we abbreviate $w(\{i,j\})$ and $\kappa(i)$ as $w_{ij}$ and $\kappa_i$, respectively. 
In particular, when the signature $\sigma\equiv+1$, the signed $p$-Laplacians reduces to the usual $p$-Laplacians on unsigned graphs. In this paper, we regard a graph $G=(V,E)$ as a signed graph $\Gamma=(G,\sigma)$ with $\sigma\equiv+1$.

Given a vertex measure $\mu:V\to\R^+$, we call a non-zero function $f\in C(V)$ an eigenfunction of $\Delta_p^{\sigma}$ if there exists a constant $\lambda\in \mathbb{R}$ such that $$\Delta_p^{\sigma}(f(i))=\lambda\mu_{i}\Psi_p(f(i)),\quad\text{for any } \,\,i\in V,$$ where we write $\mu_i$ instead of  $\mu(i)$ for short.  We call this constant $\lambda$ an eigenvalue, and $(\lambda, f)$ an eigenpair, of $\Delta_p^{\sigma}$.


From now on, whenever we speak of a signed graph $\Gamma=(G,\sigma)$, we indicate a signed graph $\Gamma=(G,\sigma)$ together with a  vertex measure $\mu$, an edge weight $w$ and a potential function $\kappa$. By a subgraph $\Gamma'=(G',\sigma)$ with $G'=(V',E')$ of $\Gamma$, we always mean the signed graph whose signature and edge weight are the restrictions of those of $\Gamma$ to $E'$, and whose vertex measure and potential function are the restrictions of those of $\Gamma$ to $V'$.

Similar as the continuous case, the index theory yields a characterization of a subset of the eigenvalues of the graph $p$-Laplacian.
For any subset $B$ of a Banach space $\mathcal{B}$, we say that $B$ is \emph{symmetric} if $B=-B$ where \[-B:=\{-x:x\in B \}.\]
\begin{defn}[index]
Let $B$ be a nonempty closed symmetric subset of a Banach space. The   \emph{index} (or Krasnoselskii genus) of $B$  is defined as follows:
\begin{equation*}
\gamma(B):=\begin{cases}
\min\limits\{k\in\mathbb{Z}^+: \,\,\text{there exists an odd continuous map}\; h: B \to \mathbb{S}^{k-1}\},  \\
\infty \qquad \text{if for any } k\in \mathbb{Z}^+,
 \text{ there does not exist any odd continuous map from } B \text{ to } \mathbb{S}^{k-1}. 
\end{cases}
\end{equation*}
\end{defn}
Let us recall some fundamental properties of the index (or Krasnoselskii genus) which can be found in books about variational methods such as 
\cite{Papageorgiou09,Rabinowitz,Struwe,Solimini}.

\begin{pro}\label{pro:index}
   Let $\mathcal{B}$ be a Banach space. We have the following properties.
	\begin{itemize}
          \item [(i)]  Let $B$ be a bounded  symmetric open set containing the origin $O$ of $\mathcal{B}$. Then we have $\gamma(\partial B)=\text{dim }\mathcal{B}$.
          \item [(ii)] 
        Let $\mathcal{C}$ be a linear subspace of $\mathcal{B}$ with $\mathrm{dim }\, \mathcal{C}=k<+\infty$. Define $p_{\mathcal{C}}:\mathcal{B}\to \mathcal{B}$ to be the projection operator onto $\mathcal{C}$. If $B$ is a closed symmetric subset of $\mathcal{B}$ with $\gamma(B)>k$, then $B\cap(I-p_{\mathcal{C}})(\mathcal{B})\neq \emptyset$.
		\item [(iii)]
   Let $H:B\to \mathbb{R}^k$ be an odd continuous map, where $B$ is a closed symmetric subset of $\mathcal{B}$ with $\gamma(B)>k$. Then we have $\gamma(H^{-1}(0))\geq \gamma(B)-k$.
  	\item [(iv)]
  If $B\subset \mathcal{B}$ is a compact symmetric set and the origin $O\notin B$, then there is a neighbourhood $N$ of $B$ such that $\overline{N}$ is a closed symmetric set with $\gamma(\overline{N})=\gamma(B)$.
  \item [(v)]
  If $B_1, B_2$ are two closed symmetric subset of $ \mathcal{B}$ with $B_1\subset B_2$, then $\gamma(B_1)\leq \gamma(B_2)$.
	\end{itemize}
\end{pro}

We define $$\mathcal{S}_p(V):=\left\{f\in C(V)\left|\;\sum_{i=1}^n\mu_{i}|f(i)|^p=1\right.\right\}$$ to be the set of functions with unit $\ell^p(\mu)$ norm
and \[\mathcal{F}_k(\mathcal{S}_p(V)):=\left\{B\subset\mathcal{S}_p: B \text{ is closed symmetric and } \gamma(B)\geq k\right\}.\] Whenever it is clear from the context, we write $\mathcal{S}_p$ and $\mathcal{F}_k(\mathcal{S}_p)$ instead of $\mathcal{S}_p(V)$ and $\mathcal{F}_k(\mathcal{S}_p(V))$  for short.
We define the Rayleigh quotient of $\Delta_p^{\sigma}$ as follows:$$\mathcal{R}_p^{\sigma}(f)=\frac{\sum_{\{i,j\}\in E} w_{ij}|f(i)-\sigma_{ij}f(j)|^p+\sum_{i=1}^n\kappa_i|f(i)|^p}{\sum_{i=1}^{n}\mu_{i}|f(i)|^p},$$ for any non-zero function $f\in C(V).$
The Lusternik-Schnirelman theory
allows us to define a sequence of \emph{variational  eigenvalues} of $\Delta_p^\sigma$ \cite{Papageorgiou09}:
$$ \lambda_k^{(p)}(\Gamma):=\min_{B\in \mathcal{F}_{k}( \mathcal{S}_p)}\max\limits_{f\in B}  \RQ_p^\sigma(f),\;\;k=1,2,\cdots,n.$$
We remark that all variational eigenvalues are eigenvalues of $\Delta_p^{\sigma}$. But there does exist eigenvalues which are not variational eigenvalues, see \cite[Theorem 6]{Amghibech}. Recall that the (variational) eigenvalues are switching invariant.
\begin{pro}{\cite[Proposition 2.5]{GLZ22+}}\label{pro:GLZswitching}
Let $(G,\sigma)$ and $(G,\tilde{\sigma})$ be two signed graphs with the same vertex measure, edge weight and potential function. Suppose that $\tilde{\sigma}$ is switching equivalent to $\sigma$ such that $\tilde{\sigma}=\sigma^\tau$ for some switching function $\tau$. Then $(\lambda,f)$ is an eigenpair of $\Delta_p^{\sigma}$ if and only $(\lambda, \tau f)$ is an eigenpair of  $\Delta_p^{\tilde{\sigma}}$. Moreover, we have for variational eigenvalues that \[\lambda_k^{(p)}(G,\sigma)=\lambda_k^{(p)}(G,\tilde{\sigma})\] for any $k=1,2,\ldots,n$.
\end{pro}

The adjacency matrix $A_{\Gamma}=(a_{ij})_{1\leq i,j\leq  n}$ of a signed graph $\Gamma=(G,\sigma)$ with $G=(V,E)$ is an $n\times n$ matrix  defined as follows
 \begin{equation*}
   a_{ij}=\left\{ \begin{aligned}
        &w_{ij}\quad &&i \sim j \text{ and }\sigma_{ij}=+1,\\
        &-w_{ij}\quad&& i\sim j \text{ and }\sigma_{ij}=-1,\\
        &0 \quad && i\nsim j.
    \end{aligned}
    \right.
\end{equation*}
Moreover, we define  the \emph{normalized adjacency matrix} $A^{\mu}_{\Gamma}$ of $\Gamma$ as 
\begin{equation}\label{eq:normAdj}
A^{\mu}_{\Gamma}=D^{-1}A_{\Gamma},\end{equation}
where $D=(d_{ij})_{1\leq i,j\leq n}$ is the diagonal matrix with $d_{ii}=\mu_i$ for $1\leq i \leq n$. 
We remark that $A_{-\Gamma}=-A_{\Gamma}$ and $A^{\mu}_{-\Gamma}=-A^{\mu}_{\Gamma}$ where $-\Gamma=(G,-\sigma)$ is the negation of $\Gamma$.

Note that when $\mu\equiv1$, the normalized  adjacency matrix reduces to the adjacency matrix.

In this paper, we always denote by $$\lambda_1(M)\leq \lambda_2(M)\leq \cdots\leq \lambda_n(M)$$  the eigenvalues of an $n\times n$ matrix $M$.  We use $n^+(M)$ (resp., $n^-(M)$) to denote the number of positive (resp., negative) eigenvalues of $M$. 

 We define an inner product  $\langle\, ,\,\rangle_{\mu}:C(V)\times C(V)\to \mathbb{R}$ as follows
 $$\langle f,g\rangle_{\mu}=\sum_{i=1}^n\mu_if(i)g(i),\qquad \text{for any }f,g\in C(V). $$
 The Rayleigh quotient $\mathcal{R}_{A^{\mu}_{\Gamma}}$ of $A^{\mu}_{\Gamma}$ is defined as
$$\mathcal{R}_{A^{\mu}_\Gamma}(f):=\frac{f^TA_\Gamma f}{\langle f,f\rangle_{\mu}}=\frac{\langle f,A_{\Gamma}^{\mu}f\rangle_{\mu}}{\langle f,f\rangle_{\mu}},$$ for any non-zero function $f\in C(V).$
We next recall the following min-max principle. For any $1\leq k \leq n$,
let $P_k$ (resp.,  $\overline{P}_k$) be the set of linear subspaces of $\mathbb{R}^n$ of dimension (resp., codimension) at least $k$. Then we have the following characterizations of the eigenvalues:

$$\lambda_k(A^{\mu}_{\Gamma})=\min_{P\in P_k}\max_{0\neq f \in P}\mathcal{R}_{A^{\mu}_{\Gamma}}(f)=\max_{P\in \overline{P}_{k-1}}\min_{0\neq f \in \overline{P}}\mathcal{R}_{A^{\mu}_{\Gamma}}(f).$$

Let us conclude this section by defining a family of maps which will be often used in the proofs of our results. For $p, q> 1$, we define the map $\phi_{p,q}:\mathcal{S}_p\to \mathcal{S}_q$ 
as follows
\begin{equation}\label{eq:phipq}
\phi_{p,q}\left(f_1,f_2,\ldots,f_n\right):=\left(|f_1|^{\frac{p}{q}}\mathrm{sign}(f_1), |f_2|^{\frac{p}{q}}\mathrm{sign}(f_2),\ldots,|f_n|^{\frac{p}{q}}\mathrm{sign}(f_n)\right),
\end{equation}
for any $f=(f_1,f_2,\ldots,f_n)\in \mathcal{S}_p$, where the function $\mathrm{sign}$ is defined as
 \begin{equation*}
   \mathrm{sign}(x)=\left\{ \begin{aligned}
        &+1 & \qquad x> 0,\\
        &0 & \qquad x=0,\\
        &-1  &  x<0.
    \end{aligned}
    \right.
\end{equation*}
Observe that $\phi_{q,p}$ is the inverse map of $\phi_{p,q}$. Indeed, $\phi_{p,q}$ is an odd homeomorphism.



\section{Monotonicity of variational eigenvalues}
\label{sec:mono}
In this section, we prove the monotonicity of variational eigenvalues with respect to $p$.  Denote the variational eigenvalues of the $p$-Laplacian of  a signed graph $\Gamma$ by $$\lambda_k^{(p)}:=\lambda_k^{(p)}(\Gamma)\,\,\text{for any}\,\,1\leq k\leq n.$$  

The following theorem is an extension of the theorems in \cite {Zhang}.
\begin{theorem}\label{thm:Monotonicity}
Let $\Gamma=(G,\sigma)$ be a signed graph. If $1< p \leq q$, $1\leq k \leq n$ and $\kappa \geq 0$, we have 
\begin{equation}\label{eq:de}
    2^{-p}\lambda_k^{(p)}\geq 2^{-q}\lambda_k^{(q)},
\end{equation}  and 
\begin{equation}\label{eq:in}
    p\left(\frac{\lambda_k^{(p)}}{\mathcal{D}}\right)^{\frac{1}{p}}\leq q\left(\frac{\lambda_k^{(q)}}{\mathcal{D}}\right)^{\frac{1}{q}},
\end{equation}
where $\mathcal{D}:=\max_{1\leq i\leq n}\frac{2\kappa_i+\sum_{j\sim i}w_{ij}}{2\mu_i}$.
\end{theorem}
Hua and Wang \cite[Proposition 3.1]{HW20} show a related monotonicity result concerning the Dirichlet $p$-Laplacian. Translating equivalently into the current setting, they prove that, for an unsigned graph with 
$\kappa\geq 0$ and $\mu_i=\kappa_i+\sum_{j\sim i}w_{ij}$ for any $i$,
\begin{equation}\label{eq:HuaWang}
   p \left(\lambda_1^{(p)}\right)^{\frac{1}{p}}\leq q \left(\lambda_1^{(q)}\right)^{\frac{1}{q}}.
\end{equation}
This is an extension of the continuous case result \cite{Lindqvist}. Note that if $\mu_i=\kappa_i+\sum_{j\sim i}w_{ij}$, then the constant $\mathcal{D}\leq 1$. Therefore, our estimate (\ref{eq:in}) improves (\ref{eq:HuaWang}).

For an unsigned graph $G=(V,E)$ with $\kappa\equiv0$ and $\mu_i=\sum_{j\sim i}w_{ij}$ for any $i$,  Dru\c{t}u and Mackay \cite[Lemma 2.7]{DM19} prove that $\lambda_2^{(p)}$ is right lower semi-continuous with respect to $p$ via establishing the following  monotonicity 
\begin{equation}\label{eq:DM}
    \left(\frac{\lambda_2^{(p)}}{|E|}\right)^{\frac{1}{p}}\leq \left(\frac{\lambda_2^{(q)}}{|E|}\right)^{\frac{1}{q}},
\end{equation}
for any $2\leq p \leq q$, where $|E|$ is the number of edges in $G$. Note that if $\mu_i=\sum_{j\sim i}w_{ij}$, then the constant $\mathcal{D}=\frac{1}{2}$. In case $2\leq p\leq q\leq \ln (2|E|)$, our estimate (\ref{eq:in}) implies $(\ref{eq:DM})$. For the case $\ln(2|E|)\leq p\leq q$, their estimate $(\ref{eq:DM})$ implies ours (\ref{eq:in}) with $k=2$.

Next, we prepare for the proof of Theorem \ref{thm:Monotonicity} several lemmas. The following two inequalities are modified from \cite[Lemma A.1]{Zhang} and \cite[Lemma 3]{Amghibech} to adapt our setting of signed graphs. 
\begin{lemma}\label{lemma:geq}
For any $t\geq 1,\,a,b\in \mathbb{R},\,\sigma\in \{-1,+1\}$, we have $$\left| |b|^t\mathrm{sign}(b)-\sigma|a|^t\mathrm{sign}(a)  \right|   \geq |b-\sigma a| \left(  \frac{|b|^t+|a|^t}{2} \right)^{1-\frac{1}{t}}.$$
\end{lemma}
\begin{proof}
When $\sigma=+1$, it is the \cite[Lemma A.1]{Zhang}. For the case of $\sigma=-1$, the proof is similar to that of  \cite[Lemma A.1]{Zhang}.
\end{proof}
\begin{lemma}\label{lemma:leq}
    For any $t\geq 1,\, a,b\in \mathbb{R},\,\sigma\in \{-1,+1\}$, we have 
    \[\left| |b|^t\mathrm{sign}(b)-\sigma|a|^t\mathrm{sign}(a)  \right|   \leq t|b-\sigma a| \left( \frac{|b|^t+|a|^t}{2} \right)^{1-\frac{1}{t}}.\]
\end{lemma}
\begin{proof}
    We only present the proof of the case $\sigma=+1$. The proof of the case $\sigma=-1$ is similar.
    
Notice that the inequality is trivial for $t=1$. We assume $t>1$ in the following.
    When $ab\geq 0$, it is proved in \cite[Lemma 3]{Amghibech}. It remains to consider the situation $ab<0$. Without loss of generality, we assume $b>0>a$. Letting $c:=|a|$, we compute by H\"older inequality 
    \begin{align}
        \frac{1}{t} \left|  |b|^t\mathrm{sign}(b)-\sigma|a|^t\mathrm{sign}(a)\right|& =\frac{1}{t}\left(b^t+c^t\right)
         =\int_{-c}^b|x|^{t-1}dx \notag\\&\leq \left(\int_{-c}^b|x|^tdx\right)^{1-\frac{1}{t}}(b+c)^{\frac{1}{t}}=\left(\frac{b^{t+1}+c^{t+1}}{t+1}\right)^{1-\frac{1}{t}}(b+c)^{\frac{1}{t}}. \label{eq:Hoelder}
         \end{align}
We claim that
\begin{equation}\label{eq:tbc}\frac{b^{t+1}+c^{t+1}}{t+1}\leq \frac{\left(b^t+c^t\right)(b+c)}{2}.
\end{equation}
In fact, it is equivalent to the inequality $\frac{1}{t+1}\left(x^{t+1}+1\right)\leq \frac{1}{2}\left(x^{t}+1\right)\left(x+1\right)$ for any $x\geq 1$. The latter can be derived from the following direct computations $$\frac{t+1}{2}\left(x^t+1\right)(x+1)-x^{t+1}-1=\frac{t-1}{2}x^{t+1}+\frac{t+1}{2}x^t+\frac{t+1}{2}x+\frac{t-1}{2}\geq 0.$$
Inserting the equation (\ref{eq:tbc}) into (\ref{eq:Hoelder}), we have $$ \frac{1}{t} \left|  |b|^t\mathrm{sign}(b)-\sigma|a|^t\mathrm{sign}(a)\right|\leq \left(\frac{b^t+c^t}{2}\right)^{1-\frac{1}{t}}(b+c)=|b-\sigma a| \left( \frac{|b|^t+|a|^t}{2} \right)^{1-\frac{1}{t}}.$$
This concludes the proof.
\end{proof}
The above two inequalities leads to the following interesting monotonicity properties of Rayleigh quotients. 

\begin{lemma}\label{lemma:RayMonotonicity}
Let $\Gamma=(G,\sigma)$ be a signed graph with the potential function $\kappa \geq 0$. For $1\leq p \leq q$ and any $f\in \mathcal{S}_q$, we have 
\begin{equation}\label{eq:deRay}
    2^{-p}\mathcal{R}_p^\sigma(\phi_{q,p}(f))\geq 2^{-q}\mathcal{R}_q^\sigma(f),
\end{equation}  and 
\begin{equation}\label{eq:inRay}
    p\left(\mathcal{R}_p^\sigma(\phi_{q,p}(f))/\mathcal{D}\right)^{\frac{1}{p}}\leq q\left(\mathcal{R}_q^\sigma(f)/\mathcal{D}\right)^{\frac{1}{q}},
\end{equation}
where $\phi_{q,p}: \mathcal{S}_q\to \mathcal{S}_p$ is the map defined in (\ref{eq:phipq}) and $\mathcal{D}=\max_{1\leq i\leq n}\frac{2\kappa_i+\sum_{j\sim i}w_{ij}}{2\mu_i}$.
\end{lemma}
\begin{proof}
Given any $f\in \mathcal{S}_q$, we compute directly that
\begin{equation*}
	\begin{aligned}
		\mathcal{R}_p^{\sigma}(\phi_{q,p}(f))&=\sum_{\{i,j\}\in E}w_{ij} \left ||f(i)|^{\frac{q}{p}}\mathrm{sign}\left(f(i)\right)-\sigma_{ij}|f(j)|^{\frac{q}{p}}\mathrm{sign}\left(f(j)\right) \right|^p +\sum_{i=1}^n \kappa_i|f(i)|^q \\
		&\geq \sum_{\{i,j\}\in E}w_{ij}|f(i)-\sigma_{ij}f(j)|^p\left| \left( \frac{|f(i)|^{\frac{q}{p}}+|f(j)|^{\frac{q}{p}}}{2} \right)^{1-\frac{p}{q}} \right|^p+\sum_{i=1}^n \kappa_i|f(i)|^q \\
		&\geq \sum_{\{i,j\}\in E}w_{ij}|f(i)-\sigma_{ij}f(j)|^p\left| \left( \frac{|f(i)|+|f(j)|}{2} \right)^{\frac{q}{p}(1-\frac{p}{q})} \right|^p+\sum_{i=1}^n \kappa_i|f(i)|^q \\
		&\geq 2^{p-q}\sum_{\{i,j\}\in E}w_{ij}|f(i)-\sigma_{ij}f(j)|^q +\sum_{i=1}^n \kappa_i|f(i)|^q   \\
		&\geq 2^{p-q}\mathcal{R}^{\sigma}_q(f).
	\end{aligned}
\end{equation*}
In the above, we apply Lemma \ref{lemma:geq} in the first inequality, and the convexity of the function $x\mapsto x^{\frac{q}{p}}$ for $x\geq 0$ in the second inequality. This proves the inequality (\ref{eq:deRay}).   

On the other hand, by applying Lemma \ref{lemma:leq} and H\"older inequality, we calculate for any $f\in \mathcal{S}_q$ that 
\begin{equation*}
	\begin{aligned}
		\mathcal{R}_p^{\sigma}(\phi_{q,p}(f))&=\sum_{\{i,j\}\in E}w_{ij} \left ||f(i)|^{\frac{q}{p}}\mathrm{sign}(f(i))-\sigma_{ij}|f(j)|^{\frac{q}{p}}\mathrm{sign}(f(j)) \right|^p+\sum_{i=1}^n\kappa_i|f(i)|^q \\
  &\leq \sum_{\{i,j\}\in E}w_{ij}\left|\frac{q}{p}|f(i)-\sigma_{ij}f(j)|\left(\frac{|f(i)|^{\frac{q}{p}}+|f(j)|^{\frac{q}{p}}}{2}\right)^{1-\frac{p}{q}}\right|^p+\sum_{i=1}^n\kappa_i|f(i)|^q\\
  & \leq\left(\frac{q}{p}\right)^{p}\left(\mathcal{R}_q^\sigma(f)\right)^{\frac{p}{q}}\left(\sum_{\{i,j\}\in E}w_{ij}\left(\frac{|f(i)|^{\frac{q}{p}}+|f(j)|^{\frac{q}{p}}}{2}\right)^{p\left(1-\frac{p}{q}\right)\alpha}+\sum_{i=1}^n\kappa_i|f(i)|^q\right)^{\frac{1}{\alpha}},
 	\end{aligned}
\end{equation*}
where $\alpha$ stands for the H\"older conjugate of $\frac{q}{p}$, i.e., $\frac{1}{\alpha}+\frac{p}{q}=1$.
Applying the convexity of the function $x\mapsto x^p$ for $x\geq 0$, we derive that
    \begin{equation*}
	\begin{aligned}
		\mathcal{R}_p^{\sigma}(\phi_{q,p}(f))& \leq \left(\frac{q}{p}\right)^p\left(\mathcal{R}_q^\sigma(f)\right)^{\frac{p}{q}}\left(  \sum_{\{i,j\}\in E   }w_{ij}   \frac{|f(i)|^q + |f(j)|^q}{2}+\sum_{i=1}^n\kappa_i|f(i)|^q\right)     ^{1-\frac{p}{q}} \\
        & =\left(\frac{q}{p}\right)^p\left( \mathcal{R}^{\sigma}_q(f) \right)^{\frac{p}{q}} \left( \sum_{i=1}^n\left(\frac{2\kappa_i+\sum_{j\sim i}w_{ij}}{2}\right)|f(i)|^q    \right)^{1-\frac{p}{q}}  \\	
        & \leq\left(\frac{q}{p}\right)^p\left( \mathcal{R}^{\sigma}_q(f) \right)^{\frac{p}{q}} \mathcal{D}^{1-\frac{p}{q}}.
	\end{aligned}
\end{equation*}
 The last inequality above is according to $\sum_{i=1}^n \mu_i|f(i)|^q=1$ and the definition of $\mathcal{D}$. This concludes the proof of (\ref{eq:inRay}).
\end{proof}
We still need the following existence result of minimizing sets.
\begin{lemma}\label{app:2}
    Let $\mathcal{R}:\mathcal{S}_p\to \mathbb{R}$ be a continuous map. Define $$a_k=\min_{B\in \mathcal{F}_k(\mathcal{S}_p)}\max_{f\in B}\mathcal{R}(f).$$  Then there always exists a minimizing set $B_0\in \mathcal{F}_k(\mathcal{S}_p)$ such that $a_k=\max_{f\in B_0}\mathcal{R}(f)$.
\end{lemma}
\begin{proof}
    Take $B^{(n)}\in \mathcal{F}_k(\mathcal{S}_p)$ such that $$a_k\leq \max_{f\in B^{(n)}}\mathcal{R}(f)\leq a_k+\frac{1}{n}.$$
    Define $\mathcal{F}(\mathcal{S}_p):=\{B\subset \mathcal{S}_p: B \text{ is closed and symmetric}\}$. 
    By compactness of the metric space $(\mathcal{F}(\mathcal{S}_p),d_H)$ where $d_H$ is the Hausdorff metric \cite{Munkres}, there exists $B'\in \mathcal{F}(\mathcal{S}_p)$ and a subsequence $\{B^{(n_m)}\}_{m=1}^\infty$ such that $B^{(n_m)}\stackrel{d_H}{\longrightarrow}B'$ as $m$ tends to infinity.
    By Proposition \ref{pro:index} $(iv)$, there exists a neighborhood $N$ of $B'$ such that $\gamma(B')=\gamma(\overline{N})$. 
    Notice that for $m$ large enough, we have $B^{(n_m)}\subset \overline{N}$. Due to  Proposition \ref{pro:index} $(v)$, we conclude $\gamma(\overline{N})\geq \gamma(B^{(n_m)})\geq k$. Therefore, we have $B'\in \mathcal{F}_k(\mathcal{S}_p)$, and, by continuity of $\mathcal{R}$, 
    $\max_{f\in B'}\mathcal{R}(f)=a_k$. That is,  the set $B'$ is what we need.
\end{proof}

Now, we are prepared to prove Theorem \ref{thm:Monotonicity}.

\begin{proof}[Proof of Theorem \ref{thm:Monotonicity}]
For any $k$, let $B_k\in \mathcal{F}_{k}(\mathcal{S}_p) $  be a minimizing set of $\lambda_k^{(p)}$, i.e., $$\lambda_k^{(p)}=\min_{B\in \mathcal{F}_k(\mathcal{S}_p)}\max_{f\in B}\mathcal{R}^{\sigma}_p(f)=\max_{f\in B_k}\mathcal{R}^{\sigma}_p(f) .$$
The existence of such $B_k$ is due to Lemma \ref{app:2}.
For $1<p\leq q$, let $\phi_{q,p}: \mathcal{S}_q\to \mathcal{S}_p$ be the map defined in (\ref{eq:phipq}).
 Since $\phi_{q,p}$ is an odd homeomorphism, we have $\phi_{q,p}^{-1}(B_k)=\phi_{p,q}(B_k)\in \mathcal{F}_k(\mathcal{S}_q)$. Then by the inequality (\ref{eq:deRay}) from Lemma \ref{lemma:RayMonotonicity}, we have 
\begin{equation*}
     \lambda_k^{(q)}=\min_{B\in \mathcal{F}_k(\mathcal{S}_q)}\max_{f\in B}\mathcal{R}_q^{\sigma}(f)\leq \max_{f\in \phi_{p,q}(B_k)}\mathcal{R}^{\sigma}_q(f)\leq 2^{q-p}\max_{g\in B_k}\mathcal{R}^{\sigma}_p(g)= 2^{q-p}\lambda_k^{(p)}.
\end{equation*} 
This proves the inequality (\ref{eq:de}).

Next, we prove the inequality (\ref{eq:in}). Let $B'_k\in \mathcal{F}_{k}(\mathcal{S}_q) $ be a minimizing set of $\lambda_k^{(q)}$, i.e., $$\lambda_k^{(q)}=\min_{B\in \mathcal{F}_k(\mathcal{S}_q)}\max_{f\in B}\mathcal{R}^{
\sigma}_q(f)=\max_{f\in B'_k}\mathcal{R}^{\sigma}_q(f) .$$ The existence of $B'_k$ follows from Lemma \ref{app:2}. 
Then, by the inequality \eqref{eq:inRay} in Lemma \ref{lemma:RayMonotonicity},  we obtain 
\begin{equation*}
\begin{aligned}
     \lambda_k^{(p)}&=\min_{B\in \mathcal{F}_k(\mathcal{S}_p)}\max_{f\in B}\mathcal{R}^{\sigma}_p(f)\leq \max_{f\in \phi_{q,p}(B'_k)}\mathcal{R}^{\sigma}_p(f)=\max_{f\in B'_k}\mathcal{R}^{\sigma}_p(\phi_{q,p}(f))\\
     &\leq\max_{f\in B'_k}  \left(\frac{q}{p}\right)^p\left( \mathcal{R}^{\sigma}_q(f) \right)^{\frac{p}{q}} \mathcal{D}^{1-\frac{p}{q}} = \left(\frac{q}{p}\right)^p\left( \lambda_k^{(q)} \right)^{\frac{p}{q}} \mathcal{D}^{1-\frac{p}{q}},
\end{aligned}
\end{equation*}
which implies the inequality (\ref{eq:in}).
\end{proof}




\section{New graph invariants}
\label{graph invariant}
In this section, we define the so-called \emph{cut-off adjacency eigenvalues} of a signed graph. We show a variational characterization of those new graph invariants and a corresponding interlacing theorem. 

\begin{defn}\label{def:cutoffadj}
    Let $\Gamma=(G,\sigma)$ be a signed graph (with
a vertex measure $\mu$, an edge weight $w$ and a potential function $\kappa$) and $\lambda_k^{(p)}$ be the $k$-th variational eigenvalue of the corresponding $p$-Laplacian. Define a sequence of numbers $\left(L_1(\Gamma), L_2(\Gamma),\ldots, L_n(\Gamma)\right)$ as follows: $$\left(L_1(\Gamma), L_2(\Gamma),\ldots, L_n(\Gamma)\right):=\left(\lim_{p\to \infty}2^{-p}\lambda_1^{(p)}(\Gamma) ,\lim_{p\to \infty}2^{-p}\lambda_2^{(p)}(\Gamma),\ldots,\lim_{p\to \infty}2^{-p}\lambda_n^{(p)}(\Gamma)\right).$$  For each $1\leq i \leq n$, we call $L_i(\Gamma)$  the $i$-th  cut-off adjacency eigenvalue of the signed graph $\Gamma$.
\end{defn}
We remark that the limits in the above definition are all well-defined. When $\kappa\geq 0$, the limit $\lim_{p\to \infty}2^{-p}\lambda_k^{(p)}$ exists due to the fact that $2^{-p}\lambda_k^{(p)}$ is monotone decreasing with respect to $p$ and $2^{-p}\lambda_k^{(p)}\geq 0$. For a general potential $\kappa$, the limits are still well defined. Indeed, the sequence $\left(L_1(\Gamma), L_2(\Gamma),\ldots, L_n(\Gamma)\right)$ does not depend on the choice of potential functions. This is due to the following proposition. 

\begin{pro}\label{lemma:not depend}
Let $\Gamma=(G,\sigma)$ be a signed graph  (with a vertex measure $\mu$, an edge weight $w$ and a potential function $\kappa$) and $\lambda_k^{(p)}$ be the $k$-th variational eigenvalue of the corresponding $p$-Laplacian. Define  $\eta_k^{(p)}$ to be the $k$-th variational eigenvalue of the $p$-Laplacian with respect to $w$, $\mu$ and $\kappa'$ where $\kappa'\equiv0$. Then we have $$\eta_k^{(p)}-\mathcal{C}\leq \lambda_k^{(p)}\leq \eta_k^{(p)}+\mathcal{C},$$
    where $\mathcal{C}:=\max_{1\leq i\leq n}\left|\frac{\kappa_i}{\mu_i}\right|$.
\end{pro}
\begin{proof}
Let $\mathcal{R}_p^{\sigma}$ $\left(\text{resp., }\overline{\mathcal{R}_p^{\sigma}}\right)$ be the Rayleigh quotient of the $p$-Laplacian with respect to $w$, $\mu$ and $\kappa$ $\left(\text{resp., }w,\,\mu\text{ and }\kappa' \right)$, i.e.,  we have
\begin{equation*}
    \mathcal{R}_p^{\sigma}(g)=\frac{\sum_{\{i,j\}\in E}w_{ij}\left|g(i)-\sigma_{ij}g(j)\right|^p+\sum_{i=1}^n\kappa_i|g(i)|^p}{\sum_{i=1}^n\mu_i|g(i)|^p},
\end{equation*}
and \begin{equation*}
    \overline{\mathcal{R}_p^{\sigma}}(g)=\frac{\sum_{\{i,j\}\in E}w_{ij}\left|g(i)-\sigma_{ij}g(j)\right|^p}{\sum_{i=1}^n\mu_i|g(i)|^p},
\end{equation*}
for any non-zero function $g: V\to \mathbb{R}$. Observe that 
\[\overline{\mathcal{R}_p^{\sigma}}(g)-\mathcal{C}\leq \mathcal{R}_p^{\sigma}(g)\leq \overline{\mathcal{R}_p^{\sigma}}(g)+\mathcal{C}.\]
Let $B_k\in \mathcal{F}_k(\mathcal{S}_p)$ be a set such that $\lambda_k^{(p)}=\max_{g\in B_k}\mathcal{R}_p^{\sigma}(g)$.  The existence of $B_k$ is guaranteed by Lemma \ref{app:2}. By definition, we compute 
\begin{equation*}
    \begin{aligned}
        \lambda_k^{(p)}&=\max_{g\in B_k}\mathcal{R}_p^{\sigma}(g)\geq \max_{g\in B_k}\overline{\mathcal{R}_p^{\sigma}}(g)-\mathcal{C}\geq \min_{B\in\mathcal{F}_k\left(\mathcal{S}_p\right)}\max_{g\in B}\overline{\mathcal{R}^{\sigma}_p}(g)-\mathcal{C}=\eta_k^{(p)}-\mathcal{C}.
    \end{aligned}
\end{equation*}
Next, we take a set $B'_k\in \mathcal{F}_k(\mathcal{S}_p)$ such that $\eta_k^{(p)}=\max_{g\in B'_k}\overline{\mathcal{R}_p^{\sigma}}(g)$. The existence of such $B'_k$ is again due to Lemma \ref{app:2}.  We compute
\begin{equation*}
\begin{aligned}
    \lambda_k^{(p)}&=\min_{B\in \mathcal{F}_k\left(\mathcal{S}_p\right)}\max_{g\in B}\mathcal{R}_p^{\sigma}(g)\leq \min_{B\in \mathcal{F}_k(\mathcal{S}_p)} \max_{g\in B}\overline{\mathcal{R}_p^{\sigma}}(g)+\mathcal{C}\leq \max_{g\in B'_k}\overline{\mathcal{R}^{\sigma}_p}(g)+\mathcal{C}=\eta_k^{(p)}+\mathcal{C}.
\end{aligned}
\end{equation*}
This completes the proof.
\end{proof}
By the above proposition, we have $$\lim_{p\to \infty}2^{-p}\eta_k^{(p)}=\lim_{p\to \infty}2^{-p}\lambda_k^{(p)},$$ for any $k$. This shows that our invariants do not depend on the choice of potential functions. For convenience, we can always assume $\kappa\equiv0$ in the definition of $(L_1(\Gamma),L_2(\Gamma),\ldots,L_n(\Gamma))$.

Whenever it is clear from the context, we write $L_k$ instead of $L_k(\Gamma)$ for short.
Observe that $L_1\leq L_2\leq \cdots\leq L_n$ since $\lambda_1^{(p)}\leq\lambda_2^{(p)} \cdots\leq \lambda_{n}^{(p)}$.

\begin{pro}\label{pro:switching-spectra}
For any two switching equivalent signed graphs 
$\Gamma=(G,\sigma)$ and $\Tilde{\Gamma}=(G,\tilde{\sigma})$, we have $(L_1(\Tilde{\Gamma}),L_2(\Tilde{\Gamma}),\ldots,L_n(\Tilde{\Gamma}))=(L_1(\Gamma),L_2(\Gamma),\ldots,L_n(\Gamma))$.
\end{pro}
\begin{proof}
    This proposition follows straightforwardly from the switching invariant property of variational eigenvalues stated in Proposition \ref{pro:GLZswitching}.
\end{proof}
Next, we present a variational characterization of the cut-off adjacency eigenvalues.

\begin{theorem}\label{thm:varitional}
   Let $\Gamma=(G,\sigma)$ be a signed graph with $G=(V,E)$. For each $1\leq k\leq n$, we have 
   \[L_k(\Gamma)=\min_{B\in \mathcal{F}_k(\mathcal{S}_q)}\max_{g\in B}\mathcal{R}^{\sigma}_{q,\infty}(g),\,\,\,\text{for any}\,\,q\geq 1,\]
where \[\mathcal{R}^{\sigma}_{q,\infty}(g):=\frac{\sum_{\{i, j\}\in E} w_{ij} \left(\max\{-g(i)\sigma_{ij}g(j), 0\}\right)^{\frac{q}{2}}}{\sum_{i=1}^n\mu_i\left|g(i)\right|^q},\]
   for any non-zero function $g: V\to \mathbb{R}$. 
\end{theorem}
A particular convenient case of the above variational characterization is provided by setting $q=2$. It is worth comparing the quotient 
\[\mathcal{R}^{\sigma}_{2,\infty}(g):=\frac{\sum_{\{i, j\}\in E} w_{ij} \max\left\{-g(i)\sigma_{ij}g(j), 0\right\}}{\sum_{i=1}^n\mu_ig(i)^2}\]
with the Rayleigh quotient of the normalized adjacency matrix $A^\mu_{-\Gamma}$ of the signed graph $-\Gamma$:
\begin{equation*}
   \mathcal{R}_{A^\mu_{-\Gamma}}(g)= \frac{\langle g, A^\mu_{-\Gamma}g\rangle_\mu}{\langle g,g\rangle_{\mu}}=\frac{-\sum_{\{i, j\}\in E}w_{ij}g(i)\sigma_{ij}g(j)}{\sum_{i=1}^n\mu_ig(i)^2}.
\end{equation*}
This suggests the name \emph{cut-off adjacency eigenvalues} for $L_k(\Gamma)$'s. More connections between the cut-off adjacency eigenvalues $(L_1(\Gamma), L_2(\Gamma),\ldots,L_n(\Gamma))$ and the eigenvalues of $A^\mu_{-\Gamma}$ will be discussed  in Section \ref{Lower bound}. 

We further remark that the quotient $\mathcal{R}^{\sigma}_{2,\infty}$ can be  viewed more or less as a vector-dependent quadratic form in some sense. 
Then it intuitively corresponds to an eigenvector-dependent eigenvalue problem, although the term `eigenvector-dependent nonlinear eigenvalue problem' has been used with totally different formulations,  purposes and  motivations in the literature (see e.g., \cite{CZBL18,TL21}).

To prove Theorem \ref{thm:varitional}, we first observe the following reformulation of (\ref{eq:deRay}) from Lemma \ref{lemma:RayMonotonicity}.
\begin{lemma}\label{app:1}
    Let $\Gamma=(G,\sigma)$ be a signed graph and $q\geq 1$ be a given number. For any $g\in \mathcal{S}_q$, the function \[p\mapsto 2^{-p}\mathcal{R}_p^\sigma(\phi_{q,p}(g))\] is monotonically non-increasing on $[1,\infty)$.   
\end{lemma}
\begin{proof}
For any $g\in \mathcal{S}_q$ and $1\leq p_1\leq p_2$, we have by definition $\phi_{q,p_2}(g)\in \mathcal{S}_{p_2}$ and $\phi_{p_2,p_1}\circ\phi_{q,p_2}=\phi_{q,p_1}$. Then, we derive by Lemma \ref{lemma:RayMonotonicity} that
\[ 2^{-p_2}\mathcal{R}_{p_2}^\sigma\left(\phi_{q,p_2}(g)\right)\leq 2^{-p_1}\mathcal{R}_{p_1}^\sigma\left(\phi_{p_2,p_1}\circ\phi_{q,p_2}(g)\right)=2^{-p_1}\mathcal{R}_{p_1}^\sigma\left(\phi_{q,p_1}(g)\right).\]
This completes the proof.
\end{proof}

\begin{proof}[Proof of Theorem \ref{thm:varitional}]
   For any $q\geq 1$ and any $1\leq k\leq n$, define  $$a_k:=\min_{B\in \mathcal{F}_k(\mathcal{S}_q)}\max_{g\in B}\mathcal{R}^{\sigma}_{q,\infty}(g).$$ 
For any $g\in \mathcal{S}_q$,  a direct computation gives  $\lim_{p\to \infty}2^{-p}\mathcal{R}^{\sigma}_p\left(\phi_{q,p}(g)\right)=\mathcal{R}^{\sigma}_{q,\infty}(g)$. Indeed, we have for each edge $\{i,j\}\in E$
\[\lim_{p\to\infty}2^{-p}\left||g(i)|^{\frac{q}{p}}\mathrm{sign}(g(i))-\sigma_{ij}|g(j)|^{\frac{q}{p}}\mathrm{sign}(g(j))\right|^p=\left(\max\{-g(i)\sigma_{ij}g(j), 0\}\right)^{\frac{q}{2}}\]
by observing the facts that
\[\lim_{p\to\infty}2^{-p}\left||g(i)|^{\frac{q}{p}}-|g(j)|^{\frac{q}{p}}\right|^p=0,\,\,
\text{and}\,\, \lim_{p\to\infty}2^{-p}\left||g(i)|^{\frac{q}{p}}+|g(j)|^{\frac{q}{p}}\right|^p=|g(i)g(j)|^{\frac{q}{2}}.\]
Therefore, the functions $2^{-p}\mathcal{R}^{\sigma}_p\left(\phi_{q,p}\right): \mathcal{S}_q\to \mathbb{R}$ converges \emph{pointwisely} to $\mathcal{R}_{q,\infty}^\sigma$ as $p$ tends to infinity. Noticing further that $2^{-p}\mathcal{R}^{\sigma}_p\left(\phi_{q,p}\right)$ is monotonically non-increasing with respect to $p$ by Lemma \ref{app:1}, and all the functions $2^{-p}\mathcal{R}^{\sigma}_p\left(\phi_{q,p}\right),\mathcal{R}_{q,\infty}^\sigma: \mathcal{S}_q\to \mathbb{R}$ are continuous, we conclude that $2^{-p}\mathcal{R}^{\sigma}_p\left(\phi_{q,p}\right): \mathcal{S}_q\to \mathbb{R}$ converges \emph{uniformly} to $\mathcal{R}_{q,\infty}^\sigma$ as $p$ tends to infinity. 
   
We first show that $L_k\leq a_k$. Let $B_0\in \mathcal{F}_k(\mathcal{S}_q)$ be a minimizing set such that $a_k=\max_{g\in B_0}\mathcal{R}^{\sigma}_{q,\infty}(g)$. The existence of such $B_0$ is shown in Lemma \ref{app:2}. By definition, we have $\phi_{q,p}(B_0)\in \mathcal{F}_k\left(\mathcal{S}_p\right)$ and  
\begin{equation*}
\begin{aligned}
    L_k&=\lim_{p\to\infty}2^{-p}\lambda^{(p)}_k=\lim_{p\to \infty}2^{-p} \min_{B\in \mathcal{F}_k(\mathcal{S}_p)}\max_{g\in B}\mathcal{R}^{\sigma}_{p}(g)\\
    &\leq\liminf_{p\to \infty}\max_{g\in \phi_{q,p}\left(B_0\right)}2^{-p}\mathcal{R}^{\sigma}_{p}(g)\\
     &=\liminf_{p\to \infty}\max_{g\in B_0}2^{-p}\mathcal{R}^{\sigma}_{p}\left(\phi_{q,p}(g) \right)\\
     &=\max_{g\in B_0}\mathcal{R}^{\sigma}_{q,\infty}(g).
\end{aligned}        
\end{equation*}
The last equality is because of the uniform convergence of $2^{-p}\mathcal{R}^{\sigma}_p(\phi_{q,p})$. This proves  $L_k\leq a_k$.

To prove $L_k\geq a_k$, let $B^{(p)}\in \mathcal{F}_k(\mathcal{S}_p)$ be a minimizing set of $\lambda_k^{(p)}$, i.e., $\lambda_k^{(p)}=\max_{g\in B^{(p)}}\mathcal{R}^{\sigma}_p(g)$. Let $C^{(p)}=\phi_{p,q}\left(B^{(p)}\right)\in\mathcal{F}_k(\mathcal{S}_q)$. Define $\mathcal{F}(\mathcal{S}_q):=\{B\subset \mathcal{S}_q: B\,\,\text{is closed and symmetric}\}$. By the compactness of $\left(\mathcal{F}(\mathcal{S}_q), d_H\right)$ with $d_H$ being the Hausdorff metric, there exists a subsequence $\{C^{(p_m)}\}_{m=1}^{\infty}$ and a set $C'\in \mathcal{F}(\mathcal{S}_q)$ such that $C^{(p_m)}\stackrel{d_H}{\longrightarrow}C'$ as $m$ tends to infinity. Employing a similar argument as in the proof of Lemma \ref{app:2}, we obtain $\gamma(C')\geq k$. Then we compute
\begin{equation*}
    \begin{aligned}
        L_k&=\lim_{p\to \infty}2^{-p}\lambda_k^{(p)}
        =\lim_{m\to \infty}2^{-p_m}\max_{g\in B^{(p_m)}}\mathcal{R}^{\sigma}_{p_m}(g)\\
        &=\lim_{m\to \infty}\max_{g\in C^{(p_m)}}2^{-p_m}\mathcal{R}^{\sigma}_{p_m}\left(\phi_{q,p_m}(g)\right)\\
        &=\max_{g\in C'}\mathcal{R}^{\sigma}_{q,\infty}(g)\geq a_k.
    \end{aligned}
\end{equation*}
The last equality above is because of the uniform convergence of $2^{-p}\mathcal{R}^{\sigma}_p(\phi_{q,p})$ and the uniform continuity of $\mathcal{R}^{\sigma}_{q,\infty}$. This concludes the proof.
\end{proof}
To conclude this section, we prove an interlacing theorem of the cut-off adjacency eigenvalues.
\begin{theorem}\label{thm:Interlacing}
    Let $\Gamma=(G,\sigma)$ be a signed graph with $G=(V,E)$. Let $\Gamma'=(G',\sigma)$ be a signed graph where $G'=(V',E')$ is an induced subgraph of $G$ obtained by removing $m$ vertices from $V$ with $1\leq m\leq n-1$. Then we have $$L_k(\Gamma)\leq L_k(\Gamma')\leq L_{k+m}(\Gamma),$$ for any $1\leq k\leq n-m$.
\end{theorem}
\begin{proof}
We only need prove the case of $m=1$.
 For $m>1$, the inequality is derived by applying recursively  the result for $m=1$.

 Suppose that $V=\{1,\ldots, n-1, n\}$ and $V'=\{1,\ldots,n-1\}=V\setminus\{n\}$.
Let $\mathcal{R}_{2,\infty}^{\sigma}$ and $\overline{\mathcal{R}_{2,\infty}^{\sigma}}$ be the cut-off Rayleigh-type quotients of $\Gamma$ and $\Gamma'$, respectively, i.e., $$\mathcal{R}^{\sigma}_{2,\infty}(g)=\sum_{\{i, j\}\in E}  \max\left\{-w_{ij}g(i)\sigma_{ij}g(j), 0\right\}, \qquad\text{for any } g\in \mathcal{S}_2(V),$$
    and 
$$\overline{\mathcal{R}^{\sigma}_{2,\infty}}(g)=\sum_{\{i, j\}\in E'}  \max\left\{-w_{ij}g(i)\sigma_{ij}g(j), 0\right\}, \qquad\text{for any } g\in \mathcal{S}_2(V').$$
      
We first prove $L_k(\Gamma)\leq L_k(\Gamma')$. Let $B'_k\in \mathcal{F}_k(\mathcal{S}_2(V'))$ be a minimizing set corresponding to $L_k(\Gamma')$, i.e., $L_k(\Gamma')=\max_{g\in B'_k}\overline{\mathcal{R}_{2,\infty}^{\sigma}}(g)$. The existence of $B'_k$ follows from  Lemma \ref{app:2}. Define the map $\psi:\mathcal{S}_2(V')\to \mathcal{S}_2(V)$ as follows
      $$\psi\left(f(1), f(2)\ldots, f(n-1)\right):= \left(f(1), f(2)\ldots, f(n-1),0\right).$$
By a direct computation, we have $\mathcal{R}^{\sigma}_{2,\infty}\left(\psi(g)\right)= \overline{\mathcal{R}^{\sigma}_{2,\infty}}(g)$ for any $g\in \mathcal{S}_2(V')$. Observing that $\psi(B'_k)\in \mathcal{F}_k(\mathcal{S}_2(V))$,  we have by Theorem \ref{thm:varitional},
\begin{equation*}
  L_k(\Gamma)=\min_{B\in \mathcal{F}_k\left(\mathcal{S}_2(V)\right)}\max_{g\in B}\mathcal{R}_{2,\infty}^{\sigma}(g)\leq \max_{g\in \psi(B'_k)}\mathcal{R}_{2,\infty}^{\sigma}(g)= \max_{f\in B'_k}\mathcal{R}_{2,\infty}^{\sigma}\left(\psi(f)\right)=\max_{f\in B'_k}\overline{\mathcal{R}_{2,\infty}^{\sigma}}(f)=L_k(\Gamma'). 
\end{equation*}
This proves $L_k\leq L'_k$.

Next, we prove $L'_k\leq L_{k+1}$. Let $B_{k+1}\in \mathcal{F}_{k+1}(\mathcal{S}_2(V))$ be a minimizing set of $L_{k+1}(\Gamma)$, i.e., $L_{k+1}(\Gamma)=\max_{g\in B_{k+1}}\mathcal{R}_{2,\infty}^{\sigma}(g)$. The existence of $B_{k+1}$ is due to  Lemma \ref{app:2}.
We define the function $H:B_{k+1}\to \mathbb{R}$ as follows 
 $$H\left(f(1), f(2)\ldots, f(n-1),f(n)\right):= f(n).$$ By Proposition \ref{pro:index} $(iii)$, we have $\gamma(H^{-1}(0))\geq k$.
 Moreover, we define  $\psi_0:H^{-1}(0)\to \mathcal{S}_2(V')$ as 
$$\psi_0\left(f(1),f(2)\ldots,f(n)\right):= \left(f(1),f(2)\ldots,f(n-1)\right).$$
Since $\psi_0$ is an odd homeomorphism from $H^{-1}(0)$ to $\psi_0\left(H^{-1}(0)\right)$, we have \[\gamma\left(\psi_{0}(H^{-1}(0))\right)=\gamma(H^{-1}(0))\geq k.\] Given any $f\in H^{-1}(0)$, a  direct computation yields $\mathcal{R}_{2,\infty}^{\sigma}(f)=\overline{\mathcal{R}_{2,\infty}^{\sigma}}\left(\psi_0(f)\right)$. By definition, we have $$L_{k+1}(\Gamma)=\max_{g\in B_{k+1}}\mathcal{R}_{2,\infty}^{\sigma}(g)\geq \max_{g\in H^{-1}(0)}\mathcal{R}_{2,\infty}^{\sigma}(g)=\max_{f\in \psi_0\left(H^{-1}(0)\right)}\overline{\mathcal{R}_{2,\infty}^{\sigma}}(f)\geq L_{k}(\Gamma').$$
This concludes the proof.    
\end{proof}

\section{Lower bounds of the $p$-Laplacian eigenvalues}\label{Lower bound}
 In this section, we derive new lower bounds for the $p$-Laplacian variational eigenvalues of signed graphs, with the help of the new invariants $(L_1,\ldots, L_n)$.

 Recall that for any signed graph $\Gamma=(G,\sigma)$, we denote by $A^{\mu}_{-\Gamma}$ the normalized  adjacency matrix (\ref{eq:normAdj}) of the signed graph $-\Gamma=(G,-\sigma)$. We estimate the cut-off adjacency eigenvalues as below.

\begin{theorem}\label{thm:L_n's antibalanced}
     Let $\Gamma=(G,\sigma)$ be a signed graph with $G=(V,E)$. For $1\leq k \leq n$, we have 
     $$L_k(\Gamma)\ge\frac{1}{2}\max_{\Gamma_0\in \mathcal{G}}\lambda_{|V_{\Gamma_0}|-n+k}(A^{\mu}_{-\Gamma_0})=\frac{1}{2}\max_{\substack{\Gamma_0\in \mathcal{G}\\|V_{\Gamma_0}|=n } }\lambda_{k}(A^{\mu}_{-\Gamma_0})$$
 where $\mathcal{G}$ is the set consisting of all subgraphs of $\Gamma$, $|V_{\Gamma_0}|$ is the number of vertices of $\Gamma_0$.
 
When $k=n$, the inequality above is an equality. Moreover, we have 
     $$L_n(\Gamma)
     =\frac{1}{2}\max_{\substack{\Gamma_0\in \mathcal{G},\,|V_{\Gamma_0}|=n \\ \Gamma_0 \,\,\text{is antibalanced} } }\lambda_{n}(A^{\mu}_{-\Gamma_0}).$$
\end{theorem}
We have the following interesting corollary of Theorem \ref{thm:L_n's antibalanced}.

\begin{cor}\label{cor:anti-balanced}
    Let $\Gamma=(G,\sigma)$ be a signed graph. For $1\leq k\leq n$, we have
    \[L_k(\Gamma)\geq \frac{1}{2}\lambda_k(A^\mu_{-\Gamma}).\]
    In particular, when the signed graph $\Gamma$ is antibalanced, the equality holds for $k=n$, i.e.,   \begin{equation}\label{eq:Ln}
    L_n(\Gamma)=\frac{1}{2}\lambda_n(A^{\mu}_{-\Gamma}).
    \end{equation}
\end{cor}
As an application of the identity \eqref{eq:Ln}, we analyze the limit of the largest eigenfunction. 
\begin{theorem}\label{thm:anti-eigen}
        Let $\Gamma=(G,\sigma)$ be a connected signed graph with $G=(V,E)$ and $\sigma \equiv -1$. For any $p>1$, let $f^{(p)}\in\mathcal{S}_p(V)$ be an eigenfunction of $\Delta_p^{\sigma}$ corresponding to $\lambda_n^{(p)}$. Then for any $i\in V$, the limit $\lim_{p\to \infty}\left|f^{(p)}(i)\right|^{\frac{p}{2}}$ exists. Moreover, the limit function $f$ defined by
        $$f(i):=\lim_{p\to \infty}\left|f^{(p)}(i)\right|^{\frac{p}{2}},\,\,\text{for all}\,\, i\in V,$$ 
        is an eigenfunction of 
       $A^{\mu}_{-\Gamma}$ corresponding to $\lambda_n(A^{\mu}_{-\Gamma})$. 
\end{theorem}

Combining Proposition \ref{lemma:not depend}, the monotonicity inequality in Theorem \ref{thm:Monotonicity} and the above Theorem \ref{thm:L_n's antibalanced} yields the following new lower bounds of $p$-Laplacian variational eigenvalues. 
\begin{theorem}\label{thm:lower bound}
    Let $\Gamma=(G,\sigma)$ be a signed graph with $G=(V,E)$. For any  $p\geq 1$ and any $1\leq k \leq n$, we have 
     $$\lambda_k^{(p)}\ge2^{p-1}\max_{\Gamma_0\in \mathcal{G}}\lambda_{|V_{\Gamma_0}|-n+k}(A^{\mu}_{-\Gamma_0})-\mathcal{C}=2^{p-1}\max_{\substack{\Gamma_0\in \mathcal{G}\\|V_{\Gamma_0}|=n } }\lambda_{k}(A^{\mu}_{-\Gamma_0})-\mathcal{C}.$$
 where $\mathcal{G}$ is the set consisting of all subgraphs of $\Gamma$, $|V_{\Gamma_0}|$ is the number of vertices of $\Gamma_0$ and $\mathcal{C}:=\max_{1\leq i\leq n}\left|\frac{\kappa_i}{\mu_i}\right|$.
\end{theorem}
\begin{remark}
The inequalities in \eqref{eq:1} from the introduction are consistent with Theorem~\ref{thm:lower bound}, since
\[
    -A^\mu_{\Gamma_0}=A^\mu_{-\Gamma_0}.
\]
\end{remark}
Next, we present the proofs for Theorem \ref{thm:L_n's antibalanced}, Corollary \ref{cor:anti-balanced} and Theorem \ref{thm:lower bound}, respectively.

\begin{proof}[Proof of Theorem \ref{thm:L_n's antibalanced}]
Since $L_k(\Gamma)$ does not depend on the choice of the potential function $\kappa$, we assume $\kappa\equiv0$ in this proof.
  We first prove that 
  \begin{equation}\label{eq:Gamma_0=n}
  \frac{1}{2}\max_{\Gamma_0\in \mathcal{G}}\lambda_{|V_{\Gamma_0}|-n+k}(A^{\mu}_{-\Gamma_0})=\frac{1}{2}\max_{\substack{\Gamma_0\in \mathcal{G}\\|V_{\Gamma_0}|=n } }\lambda_{k}(A^{\mu}_{-\Gamma_0}).   
  \end{equation} 
The inequality $$\frac{1}{2}\max_{\Gamma_0\in \mathcal{G}}\lambda_{|V_{\Gamma_0}|-n+k}(A^{\mu}_{-\Gamma_0})\geq \frac{1}{2}\max_{\substack{\Gamma_0\in \mathcal{G}\\|V_{\Gamma_0}|=n } }\lambda_{k}(A^{\mu}_{-\Gamma_0})$$ is straightforward.
For the converse, let $\Gamma'\in \mathcal{G}$ be a subgraph satisfying $$\frac{1}{2}\lambda_{|V_{\Gamma'}|-n+k}(A^{\mu}_{-\Gamma'})=\frac{1}{2}\max_{\Gamma_0\in \mathcal{G}}\lambda_{|V_{\Gamma_0}|-n+k}(A^{\mu}_{-\Gamma_0}).$$
Define a new signed graph $\Gamma''\in \mathcal{G}$ to be the union of the graph $\Gamma'$ together with the empty graph $K:=(V\setminus V_{\Gamma'},\emptyset)$. Then we have $$\frac{1}{2}\lambda_{|V_{\Gamma'}|-n+k}(A^{\mu}_{-\Gamma'})\leq \frac{1}{2}\lambda_{k}(A^{\mu}_{-\Gamma''})\leq \frac{1}{2}\max_{\substack{\Gamma_0\in \mathcal{G}\\|V_{\Gamma_0}|=n } }\lambda_{k}(A^{\mu}_{-\Gamma_0}),$$  by the min-max theorem. 
In consequence, we complete the proof of the equality~\eqref{eq:Gamma_0=n}.

    We next prove that for any  subgraph $\Gamma_0\in\mathcal{G}$ with $|V_{\Gamma_0}|=n$, there always holds $L_k(\Gamma)\geq \frac{1}{2}\lambda_{k}(A^{\mu}_{-\Gamma_0})$. Let $\Gamma_0=(G_0,\sigma)$ be a subgraph of $\Gamma$ with $G_0=(V_{\Gamma_0},E_{\Gamma_0})$ and $|V_{\Gamma_0}|=n$. Let $\{f_i\}_{i=1}^{n}$ be the orthonormal eigenfunctions of $A^{\mu}_{-\Gamma_0}$ with respect to the inner product $\langle\,,\,\rangle_{\mu}$, where $f_i\in \mathcal{S}_2(V_{\Gamma_0})$ is corresponding to $\lambda_i(A^{\mu}_{-\Gamma_0})$. 
    Define $W$ to be the linear space spanned by the $n-k+1$ functions $\{f_i\}_{i=k}^{n}$.  By Lemma \ref{app:2}, there exists a minimizing set $B_k\in \mathcal{F}_k(\mathcal{S}_2(V))$ of $L_k(\Gamma)$. By Proposition \ref{pro:index} $(ii)$, we get $W\cap B_k\neq \emptyset$. Picking a function $g\in W\cap B_k$, we  obtain
 \begin{equation*}
 \begin{aligned}
          L_k(\Gamma)&=\max_{f\in B_k}\mathcal{R}_{2,\infty}^{\sigma}(f)\geq \mathcal{R}_{2,\infty}^{\sigma}(g)=\sum_{\{i,j\}\in E}\max\{-w_{ij}g(i)\sigma_{ij}g(j),0\}\\
          &\geq \sum_{\{i,j\}\in E_{\Gamma_0}}\max\{-w_{ij}g(i)\sigma_{ij}g(j),0\}\geq \sum_{\{i,j\}\in E_{\Gamma_0}}-w_{ij}g(i)\sigma_{ij}g(j)\\
          &=\frac{1}{2}g^T\left(A_{-\Gamma_0}\right)g\geq\frac{1}{2}\lambda_{k}(A^{\mu}_{-\Gamma_0}).  
 \end{aligned}
 \end{equation*}
This proves the first inequality. 

For the case of  $k=n$, 
we derive from the first inequality that 
 $$L_n(\Gamma)\geq \frac{1}{2}\max_{\Gamma_0\in \mathcal{G}}\lambda_{|V_{\Gamma_0}|}(A^{\mu}_{-\Gamma_0})=\frac{1}{2}\max_{\substack{\Gamma_0\in \mathcal{G}\\|V_{\Gamma_0}|=n } }\lambda_{n}(A^{\mu}_{-\Gamma_0})\geq \frac{1}{2}\max_{\substack{\Gamma_0\in \mathcal{G},\,|V_{\Gamma_0}|=n \\ \Gamma_0 \text{\,is antibalanced} } }\lambda_{n}(A^{\mu}_{-\Gamma_0}).$$
It remains to prove
  $$L_n(\Gamma)\leq \frac{1}{2}\max_{\substack{\Gamma_0\in \mathcal{G},\,|V_{\Gamma_0}|=n \\ \Gamma_0 \text{\,is antibalanced} } }\lambda_{n}(A^{\mu}_{-\Gamma_0}).$$
Let $h\in \mathcal{S}_2(V)$ be a function such that $L_{n}(\Gamma)=\mathcal{R}^{\sigma}_{2,\infty}(h)$. Define a new antibalanced subgraph $\Gamma_1=(G_1,\sigma)$ with $G_1=(V,E_1)$ where $$E_1=\{\{i,j\}\in E:\, h(i)\sigma_{ij}h(j)<0\}.$$
Then, we have $$\lambda_{n}(A^{\mu}_{-\Gamma_1})\geq h^TA_{-\Gamma_1}~h\geq \sum_{\{i,j\}\in E_1}-2w_{ij}h(i)\sigma_{ij}h(j)=2\mathcal{R}_{2,\infty}^{\sigma}(h)=2L_n(\Gamma).$$
This concludes the proof.
\end{proof}

Let $\Gamma=(G,\sigma)$ be a balanced graph. For every subgraph $\Gamma'$ of $\Gamma$, $\Gamma'$ is antibalanced if and only if $\Gamma'$ is bipartite. This implies the following corollary.
\begin{cor}\label{cor:balanced}
    Let $\Gamma=(G,\sigma)$ be a balanced graph, then $$L_n(\Gamma)=\frac{1}{2}\max_{\substack{\Gamma'\in \mathcal{G}_0\\|V_{\Gamma'}|=n } }\lambda_{n}(A^{\mu}_{-\Gamma'}),$$
where $\mathcal{G}_0$ is the set consists of all bipartite subgraphs of $\Gamma$.
\end{cor}

\begin{proof}[Proof of Corollary \ref{cor:anti-balanced}]
Similarly as the proof of Theorem \ref{thm:L_n's antibalanced}, we assume $\kappa\equiv0$. 

The estimate $L_k(\Gamma)\geq \frac{1}{2}\lambda_k(A^\mu_{-\Gamma})$ follows directly from Theorem \ref{thm:L_n's antibalanced}. To obtain (\ref{eq:Ln}), it remains to prove $L_n(\Gamma)\leq \frac{1}{2}\lambda_n(A^{\mu}_{-\Gamma}).$

By Proposition \ref{pro:switching-spectra}, without loss of generality, we can assume $\sigma\equiv -1$.
Let $\Gamma_0=(G_0,\sigma)$ with $G_0=(V_0,E_0)$ be any connected subgraph of $\Gamma$. Let $f_0\in \mathcal{S}_2(V_0)$ be an eigenfunction corresponding to the eigenvalue $\lambda_n(A^{\mu}_{-\Gamma_{0}})$. 
Since all the entries of $A^{\mu}_{-\Gamma_{0}}$ are non-negative, by Perron-Frobenius theorem, we can assume $f_0(i)>0$ for any $i\in V_0$. Define a new function $f\in \mathcal{S}_2(V) $ as follows:
\begin{equation*}
   f(i)=\left\{ \begin{aligned}
        &f_0(i) & \qquad i\in V_0,\\
        &0  &  \qquad i\in V\setminus V_0.
    \end{aligned}
    \right.
\end{equation*}
Then we have by direct computations 
\begin{equation*}
	\begin{aligned}
	 \frac{1}{2}\lambda_n(A^{\mu}_{-\Gamma_0})=\frac{1}{2}f_0^T\left(A_{-\Gamma_0}\right)f_0=&\sum_{\{i,j\}\in E_0}w_{ij}f_0(i)f_0(j) \\\leq & \sum_{\{i,j\}\in E}w_{ij}f(i)f(j)= \frac{1}{2}f^T\left(A_{-\Gamma}\right)f\leq \frac{1}{2}\lambda_n(A^{\mu}_{-\Gamma}).
	\end{aligned}
\end{equation*} 
By arbitrariness of $\Gamma_0$, we have $$L_n(\Gamma)=\frac{1}{2}\max_{\substack{\Gamma_0\in \mathcal{G},\,|V_{\Gamma_0}|=n \\ \Gamma_0 \text{\,is antibalanced} } }\lambda_{n}(A^{\mu}_{-\Gamma_0})\leq\frac{1}{2}\lambda_n(A^{\mu}_{-\Gamma}).$$ This concludes the proof.
\end{proof}
Next, we recall the following Perron-Frobenius type result from  \cite{GLZ22+}, which will be used to prove Theorem  \ref{thm:anti-eigen}.

\begin{theorem}[{\cite[Theorem 9]{GLZ22+}}]\label{thm:pr}
 Let $\Gamma=(G,\sigma)$  be a connected signed graph with $G=(V,E)$ and $\sigma\equiv -1$. Assume that $f$ is an eigenfunction of  $\Delta_p^\sigma$ corresponding to  $\lambda_n$. Then, we have:
	\begin{itemize}
          \item [(i)] $f$ is either positive on all vertices or negative on all vertices;
          \item [(ii)] 
          Let $g$ be another eigenfunction of  $\Delta_p^\sigma$ corresponding to  $\lambda_n$. Then there must exist $c\in \mathbb{R}\setminus \{0\}$ such that $g=cf$.   
		\item [(iii)]
  Given any other eigenfunction $g$ corresponding to an eigenvalue $\lambda$, if $g(i)>0$ for any $i\in V$ or $g(i)<0$ for any $i\in V$, we have $\lambda=\lambda_n$.
	\end{itemize}
\end{theorem}


\begin{proof}[Proof of Theorem \ref{thm:anti-eigen}]
   
For any $p$, define $f_2^{(p)}:V\to \mathbb{R}$ as follows $$f_2^{(p)}(j):=\left| f^{(p)}(j) \right|^{\frac{p}{2}},\,\,\text{for all} \,\, j\in V.$$ We have $f_2^{(p)}\in \mathcal{S}_2(V)$ by definition. By compactness of $\mathcal{S}_2(V)$, we can take a function $g\in \mathcal{S}_2(V)$ and a subsequence $\{f_2^{(p_k)}\}_{k=1}^{\infty}$ such that for any $j\in V$,  $\lim_{k\to \infty}f_2^{(p_k)}(j)=g(j)$. Next, we get from Theorem \ref{thm:pr} that $\left|f^{(p)}\right|$ is also an eigenfunction of $\Delta_p^{\sigma} $ corresponding to $\lambda_n^{(p)}$.
By Corollary \ref{cor:anti-balanced}, we have 
\begin{equation*}
	\begin{aligned}
	& \frac{1}{2}\lambda_n(A^{\mu}_{-\Gamma})=L_n(\Gamma)=\lim_{p\to \infty}2^{-p}\lambda_n^{(p)}= \lim_{k\to \infty}2^{-p_k}\lambda_n^{(p_k)}= \lim_{k\to \infty}2^{-p_k}\mathcal{R}_{p_k}^{\sigma}\left(|f^{(p_k)}|\right)\\	
        & =\lim_{k\to \infty}\sum_{\{i,j\}\in E}2^{-p_k}w_{ij}\left||f^{(p_k)}(i)|+|f^{(p_k)}(j)| \right|^{p_k} \\
        & =\lim_{k\to \infty}\sum_{\{i,j\}\in E}2^{-p_k}w_{ij}\left|\left|f_2^{(p_k)}(i)\right|^{\frac{2}{p_k}}+\left|f_2^{(p_k)}(j)\right|^{\frac{2}{p_k}} \right|^{p_k}  \\
        &=\sum_{\{i,j\}\in E}\lim_{k\to \infty}2^{-p_k}w_{ij}\left|\left|f_2^{(p_k)}(i)\right|^{\frac{2}{p_k}}+\left|f_2^{(p_k)}(j)\right|^{\frac{2}{p_k}} \right|^{p_k}   \\
        &=\sum_{\{i,j\}\in E}w_{ij}\left|g(i)g(j) \right|=\frac{1}{2}g^T\left(A_{-\Gamma}\right)g\leq \frac{1}{2}\lambda_n(A^{\mu}_{-\Gamma}).	\end{aligned}
\end{equation*} 
Then the above inequality is an equality.  By Perron-Frobenius theorem, we  obtain that $g$ is an eigenfunction of $A^{\mu}_{-\Gamma}$.

For any subsequence $\{f^{(p_l)}\}_{l=1}^{\infty}$ of $\{f^{(p)}\}_{p\geq 1}$, we can do as above to take a subsequence $\{f^{(p_{l_m})}\}_{m=1}^{\infty}$ which satisfies $\lim_{m\to \infty}\left|f^{(p_{l_m})}\right|^{\frac{p_{l_m}}{2}}=g$, where the
function $g$ is defined as above. This proves that for any $i\in V$, $\lim_{p\to \infty}\left|f^{(p)}(i)\right|^{\frac{p}{2}}$ exists and equals $g(i)$. Theorem \ref{thm:anti-eigen} follows from the fact that $g$ is an eigenfunction of $A_{-\Gamma}^{\mu}$.
\end{proof}

Since the all-negative signature on a bipartite graph is switching equivalent to the all-positive signature, we derive the following corollary. 

\begin{cor}  
    Let $G=(V_1\cup V_2,E)$ be a bipartite graph, where $(V_1,V_2)$ is  its bipartition. Let $f^{(p)}$ be an eigenfunction of $\Delta_p^{\sigma}$  corresponding to $\lambda_n^{(p)}$. Then for any $i\in V$, the limit  $\lim_{p\to \infty}\left|f^{(p)}(i)\right|^{\frac{p}{2}}$ exists and the function defined as 
    \begin{equation*}
   f(i):=\left\{ \begin{aligned}
        & \lim_{p\to \infty}\left|f^{(p)}(i)\right|^{\frac{p}{2}} & \qquad i\in V_1,\\
        & -\lim_{p\to \infty}\left|f^{(p)}(i)\right|^{\frac{p}{2}}  &  \qquad i\in V_2,
    \end{aligned}
    \right.
\end{equation*}
   is an eigenfunction of $A_{\Gamma}^{\mu}$ corresponding to $\lambda_n(A_{\Gamma}^{\mu})$.
\end{cor}
\begin{proof}
    Define a switching function $\tau:V\to \{+1,-1\}$  by
        \begin{equation*}
   \tau(i)=\left\{ \begin{aligned}
        &+1 & \qquad i\in V_1,\\
        &-1  &  \qquad i\in V_2.
    \end{aligned}
    \right.
\end{equation*}
Regard the graph $G=(V_1\cup V_2,E)$ as a signed graph $\Gamma=(G,\sigma)$ with $\sigma\equiv+1$. Then we use $\tau $ to switch $\Gamma$ to an anti-balanced graph $\Gamma^{\tau}$. By Proposition \ref{pro:switching-spectra}, we obtain that $(\lambda_n^{(p)},\tau f_n^{(p)})$ is an eigenpair of $\Delta_p^{\sigma^{\tau}}$. By Theorem \ref{thm:anti-eigen}, we have $\lim_{p\to \infty}\left|\tau f^{(p)}\right|^{\frac{p}{2}}=g$, where $g$ is an eigenfunction of $A_{\Gamma}^{\mu}$  corresponding to $\lambda_n(A_{\Gamma}^{\mu})$. It is equivalent to the conclusion of this corollary.
\end{proof}

To conclude this section, we give the following lower bound of signed multi-way Cheeger constants by certain quantities corresponding to the adjacency spectrum. We first introduce the following notations. Given two subsets $V_1,V_2\subset V$, we denote by 
\begin{equation*}
    |E^{\pm }(V_1,V_2)|:=\sum_{x\in V_1}\sum_{\substack{y\in V_2\\ \sigma_{xy}=\pm 1}}w_{xy},\quad  |\partial V_1|:=\sum_{x\in V_1}\sum_{y\not\in V_1}w_{xy},\quad\text{and}\,\quad\mathrm{vol}_\mu(V_1):=\sum_{x\in V_1}\mu(x).
\end{equation*}
In the case $V_1=V_2$, we write $|E^{\pm }(V_1)|=|E^{\pm }(V_1,V_1)|$ for short. We also introduce the following notation:
\[\beta^{\sigma}(V_{1},V_{2}):=\frac{2|E^+(V_{1},V_{2})|+|E^-(V_{1})|+|E^-(V_{2})|+|\partial(V_{1}\cup V_{2})|}{\mathrm{vol}_{\mu}(V_{1}\cup V_{2})},\]

\begin{defn}\cite[Definition 3.2]{AtayLiu}
Let  $\Gamma=(G,\sigma)$ be a signed graph with $G=(V,E)$. For any integer $1\leq k\leq n$, the $k$-way signed Cheeger constant $h_k^{\sigma}$ of $\Gamma$ is defined as 
\begin{equation*}
    h_k^{\sigma}:=\min_{\{(V_{2i-1}, V_{2i})\}_{i=1}^k}\max_{i=1,\ldots,k}\beta^{\sigma}(V_{2i-1},V_{2i}),
\end{equation*}
where the minimum is taken over all possible $k$-sub-bipartitions, i.e., $(V_{2i-1}\cup V_{2i})\bigcap (V_{2j-1}\cup V_{2j})=\emptyset$ for any $i\neq j$, and $V_{2l-1}\cup V_{2l}\neq \emptyset$, $V_{2l-1}\cap V_{2l}=\emptyset$ for any $l$.
\end{defn}

\begin{cor} Let $\Gamma=(G,\sigma)$ be a signed graph with $G=(V,E)$ and $\kappa\equiv0$. Then, 
\[h_k^{\sigma}\ge 2L_k\ge \max_{\substack{\Gamma_0\in \mathcal{G}\\|V_{\Gamma_0}|=n } }\lambda_{k}(A^{\mu}_{-\Gamma_0}),\;\;k=1,\cdots,n.\]
\end{cor}
\begin{proof}
This corollary is an immediate consequence of  \cite[Theorem 8]{GLZ22+} together with Theorem \ref{thm:lower bound}.
\end{proof}
\section{Inertia bounds}\label{sec:inertia}

Recall that in a graph $G=(V,E)$, an independent set is a subset of $V$ such that any two vertices in this set are not adjacent. An independent set is called maximum if its cardinality is maximum among all independent sets of the graph $G$. An edge cover of a graph $G$ is a subset of $E$ such that every vertex of $G$ is incident to at least one edge in this set. If the cardinality of an edge cover is minimum  in all edge covers of the graph $G$, we say this edge cover is minimum. For more basics about independent sets and edge covers, we refer to the book \cite{West}.

It is an important problem to estimate the cardinality of a maximum
independent set and  a minimum  edge cover of a given graph. In 1971, D. M. Cvetkovi$\acute{c}$ used the number of non-positive eigenvalues and non-negative eigenvalues of its signed weight  matrices to derive an upper bound of the  cardinality of a maximum independent set \cite{Cvetkovic},  which is well-known as the inertia bound or Cvetkovi$\acute{c}$ bound. In this section, we use our new invariants to estimate this two graph invariants. We will prove that our estimate on the cardinality of a maximum independent set implies the Cvetkovi$\acute{c}$ bound, and our estimate on the cardinality of a minimum edge cover leads to an interesting asymptotic analysis of the graph $p$-Laplacian eigenvalues as $p$ tends to infinity. 

Recall that $n^+(M)\left(\text{resp., } n^-(M) \right)$ denotes the number of positive (resp., negative) eigenvalues of the matrix $M$. We first state  the Cvetkovi$\acute{c}$ bound. 
\begin{theorem}[Cvetkovi\'{c} bound]\label{thm:Inertia bound}
    Let $G=(V,E)$ be a graph with $n=|V|$. Define $$\mathcal{M}:=\{M: M=(m_{ij})_{1\leq i,j\leq n} \text{ is a symmetric matrix such that $m_{ij}=0$ if $i\nsim j$}\}.$$ 
    Then the cardinality $\alpha$ of a maximum independent set of $G$ satisfies
    $$\alpha\leq \min_{M\in \mathcal{M}}\min\{n-n^+(M),n-n^-(M)\}.$$
\end{theorem}
Note that in the above theorem, we have for any $M=(m_{ij})_{1\leq i,j\leq n}\in \mathcal{M}$, the diagonal entries $m_{ii}=0$ for any $i\in V$, and the off-diagonal entries $m_{ij}$ can be either positive, zero or negative. 
This is different from the adjacency matrix. We can reformulate the Cvetkovi\'{c} bound as
\begin{align*}  \alpha\leq &\min_{M\in \mathcal{M}}\min\{\# \{k: \lambda_k(M)\leq 0\},\# \{k: \lambda_k(M)\geq 0\}\}\\
=& \min_{M\in \mathcal{M}}\min\{\# \{k: \lambda_k(M)\leq 0\},\# \{k: \lambda_k(-M)\leq 0\}\}.
\end{align*}

 The Cvetkovi\'{c} bound is tight for all graphs with no more than $10$ vertices, vertex-transitive graphs with no more than $12$ vertices, perfect graphs and some families of Cayley and strongly regular graphs \cite{Elzinga07,Elzinga10,Rooney}. But it is not tight for the Paley graph on 17 vertices \cite{Sinkovic}. Since our bound in Theorem \ref{thm:Inertia vertices} below is better than Cvetkovi\'{c} bound, it is interesting to ask whether our bound is tight for the Paley graph on 17 vertices or not.

\begin{theorem}\label{thm:Inertia vertices}
    Let $G=(V,E)$ be a graph. Denote the cardinality of a maximum independent set of $G$ by $\alpha$. For any edge weight $w$ and vertex measure $\mu$, we have $$\alpha\leq  \min_\sigma\#\{k:L_k(\Gamma^{\sigma})=0\}, $$
where the minimum is taken over all possible signatures $\sigma: E\to \{+1,-1\}$ and $\Gamma^{\sigma}=(G,\sigma)$.
\end{theorem}
\begin{theorem}\label{thm:Inertia edge}
    Let $G=(V,E)$ be a graph without isolate vertices. Denote the cardinality of a minimum edge cover of $G$ by $\beta$. For any edge weight $w$ and vertex measure $\mu$, we have $$\beta\geq  \max_\sigma\#\{k:L_k(\Gamma^{\sigma})=0\}, $$\
where the maximum is taken over all possible signatures $\sigma: E\to \{+1,-1\}$ and  $\Gamma^{\sigma}=(G,\sigma)$.
\end{theorem}
Recall that the cut-off adjacency eigenvalues $L_k, 1\leq k\leq n$ are always nonnegative and do not depend on the choices of the potential functions. One of the advantages  of our bounds is that the number of zero cut-off adjacency eigenvalues does not depend on the choices of edge weights and  vertex measures (see the following Proposition \ref{pro:weight and measure}). So we can arbitrarily choose a convenient edge weight and a vertex measure in applying our bounds to estimate the cardinalities of a maximum independent set and a minimum edge cover of a graph $G$.
\begin{pro}\label{pro:weight and measure}
    Let $\Gamma=(G,\sigma)$ be a signed graph.  For $1\leq i\leq n$, let $L_i$ (resp., $L'_i$) be the cut-off adjacency eigenvalues with respect to the edge weight $w$ (resp., $w'$) and vertex measure $\mu$ (resp., $\mu'$). We have 
    $$\#\{k:L_k=0\}=\#\{k:L_k'=0\}.$$  
\end{pro}

\begin{proof}
   We only need prove  $\#\{k:L_k=0\}\geq \#\{k:L_k'=0\}$ for arbitrarily chosen $w,w',\mu$ and $\mu'$. Without loss of generality, we assume the potential function $\kappa\equiv0$.

Denote by $\lambda_k^{(p)}$ (resp., $\eta_k^{(p)}$)  the $k$-th variational eigenvalue of the $p$-Laplacian with respect to $w$ (resp., $w'$) and $\mu$ (resp., $\mu'$). Let $\mathcal{S}_p$ (resp., $\mathcal{S}_p'$) be the space of functions $f\in C(V)$ with $\sum_{i=1}^n|f(i)|^p\mu_i=1$ (resp.,  $\sum_{i=1}^n|f(i)|^p\mu'_i=1$). Then the map $\psi: \mathcal{S}'_p\to \mathcal{S}_p$ defined via
$\psi(f):=f/\left(\sum_{i=1}^n|f(i)|^p\mu_i\right)^{\frac{1}{p}}$
is an odd homeomorphism. Let us set 
\[m:=\#\{k:L_k'=0\},\,\,C_1:=\max_{\{i,j\}\in E}w_{ij}/w'_{ij}\,\,\text{and}\,\,C_2:=\min_{1 \leq i \leq n}\mu_i/\mu'_i.\] For any $p>1$, let $B_p\in \mathcal{F}_m(\mathcal{S}'_p)$ be a minimizing set of $\eta_m^{(p)}$ whose existence is guaranteed by Lemma \ref{app:2}. We compute
\begin{equation*}
    \begin{aligned}
       \lambda_m^{(p)}&=\min_{B\in\mathcal{F}_m(\mathcal{S}_p)}\max_{g\in B}\mathcal{R}_p^{\sigma}(g)\leq \max_{g\in \psi(B_p)}\frac{\sum_{\{i,j\}\in E}w_{ij}|g(i)-\sigma_{ij}g(j)|^p}{\sum_{i=1}^n\mu_i|g(i)|^p}\\
       &\leq\max_{f\in B_p}\frac{\sum_{\{i,j\}\in E}C_1w'_{ij}|f(i)-\sigma_{ij}f(j)|^p}{\sum_{i=1}^nC_2\mu'_i|f(i)|^p}\\ &=\frac{C_1}{C_2}\eta^{(p)}_m. 
    \end{aligned}
\end{equation*}
Then we have
$$L_m=\lim_{p\to \infty}2^{-p}\lambda_m^{(p)}\leq \lim_{p\to \infty}\frac{C_1}{C_2}2^{-p}\eta_m^{(p)}=\frac{C_1}{C_2}L'_m=0.$$
So we have $\#\{k:L_k=0\}\geq \#\{k:L_k'=0\}$. This concludes the proof.
\end{proof}
Next, we prove the main results, Theorem \ref{thm:Inertia vertices} and Theorem \ref{thm:Inertia edge}, of the section.
\begin{proof}[Proof of Theorem \ref{thm:Inertia vertices}]
Given any edge weight $w$ and vertex measure $\mu$, it is sufficient to prove that $L_{\alpha}(\Gamma^{\sigma})=0 $ with $\Gamma^{\sigma}=(G,\sigma)$ for any signature $\sigma$.

     Without loss of generality, we assume that $\{1,2,\ldots,\alpha\}\subset V$ is a maximum independent set. Let $W$  be the linear space spanned by the functions $f_1,\ldots f_{\alpha}$, where $f_i:V\to \mathbb{R}$ is equal to $1$ on the vertex $i$ and $0$ elsewhere. By Proposition \ref{pro:index} $ (i)$, we have $\gamma(W\cap \mathcal{S}_2)=\alpha$. For any $f\in W\cap \mathcal{S}_2$, we derive $$\mathcal{R}^{\sigma}_{2,\infty}(f):=\sum_{\{i, j\}\in E}\max\left\{-w_{ij}f(i)\sigma_{ij}f(j), 0\right\}=0.$$ 
     Hence, we compute by applying Theorem \ref{thm:varitional}
  $$L_{\alpha}(\Gamma^{\sigma})=\min_{B\in \mathcal{F}_{\alpha}(\mathcal{S}_2)}\max_{f\in B}\mathcal{R}^{\sigma}_{2,\infty}(f)\leq \max_{f\in W\cap \mathcal{S}_2}\mathcal{R}^{\sigma}_{2,\infty}(f)=0.$$
This concludes the proof. 
\end{proof}
To prove Theorem \ref{thm:Inertia edge}, it is enough to show the following lemma in which the graph is allowed to have isolated vertices. Indeed, Theorem \ref{thm:Inertia edge} follows from the fact that the following lemma holds for any signature of $G$.
\begin{lemma}\label{lemma:edge cover}
     Let $\Gamma=(G,\sigma)$ be a signed graph with $G=(V,E)$, and $\Gamma'=(G',\sigma)$ be a signed graph where $G'=(V',E')$ is an induced subgraph of $\Gamma$ obtained by removing $k$ vertices from $G$. If the cardinality  of a minimum  edge cover of graph $G'$ is $\beta'$, then we have 
     \[L_{{\beta}'+k+1}(\Gamma)\geq\frac{1}{2}\frac{w_{\min}}{\mu_{\max}}>0,\]
     where $w_{\min}:=\min_{e\in E}w_e$ and $\mu_{\max}:=\max_{i\in V}\mu_i$ are the minimal edge weight and the maximal vertex measure, respectively.
\end{lemma}
\begin{proof}
      Without loss of generality, we assume  $G'$ is induced by $V\setminus \{1,2,\ldots,k\}$ and  $\{e_1,e_2,\ldots,e_{\beta'}\}$ is a minimum  edge cover of the graph $G'$ where $e_i:=\{x_i,y_i\}$. 
      Let $B_0\in\mathcal{F}_{k+\beta'+1}(\mathcal{S}_2(V))$ be a minimizing set of $L_{k+\beta'+1}$. The existence of $B_0$ is due to the Lemma \ref{app:2}. Then, we define an odd continuous map $\varphi:B_0 \to \mathbb{R}^{\beta'+k}  $ by 
      \begin{equation*}
          \begin{aligned}
            \left(f(1),f(2),\ldots,f(n)\right)\mapsto \left(f(1), f(2) ,\ldots,f(k),f(x_1)+\sigma_{x_1y_1}f(y_1),\ldots, f(x_{\beta'})+\sigma_{x_{\beta}y_{\beta'}}f(y_{\beta'})  \right).
          \end{aligned}
      \end{equation*}
           Using Proposition \ref{pro:index} $(iii)$, we  get $\gamma(B_0\cap \varphi^{-1}(0))\geq \beta'+k+1-\beta'-k=1$. Hence, the set $B_0\cap \varphi^{-1}(0)$ is nonempty. Let us pick any $g\in B_0\cap \varphi^{-1}(0)$. By definition of the map $\varphi$, we have $g(i)=0$ for any $1\leq i \leq k$ and $g(x_j)+\sigma_{x_jy_j}g(y_j)=0$ for any $1\leq j \leq \beta'$.

 Construct a new graph $G''=(V'',E'')$ where $V''=V\setminus \{1,2,\ldots,k\}$ and $E''=\{e_1,e_2,\ldots,e_{\beta'}\}\subset E$. Then, we have $|g(i)|=|g(j)|$ for any two vertices $i$ and $j$ lying in the same connected component of $G''$. Assume that $G''$ has $l$ connected components, denoted by $\{G_m:=(V_m,E_m)\}_{m=1}^l$.  For $1\leq m \leq l$, we arbitrarily pick one vertex $z_m$ from $V_m$. 
 Then, we compute
 \begin{equation*}
 \begin{aligned}
          \mathcal{R}^{\sigma}_{2,\infty}(g)&=\sum_{\{i, j\}\in E}\max\{-w_{ij}g(i)\sigma_{ij}g(j), 0\}\\
          &\geq \sum_{\{i, j\}\in E''}\max\{-w_{ij}g(i)\sigma_{ij}g(j), 0\} \\
          &=\sum_{m=1}^lw(E_m)|g(z_m)|^2=\frac{\sum_{m=1}^l w(E_m)|g(z_m)|^2}{\sum_{m=1}^l \mu(V_m)|g(z_m)|^2}\geq C_0:= \min_{1\leq m\leq l}\frac{w(E_m)}{\mu(V_m)},
 \end{aligned}
 \end{equation*}
 where $w(E_{m}):=\sum_{e\in E_m}w_e$ and $\mu(V_m):=\sum_{i\in V_m}\mu_i$.
Applying Theorem \ref{thm:varitional}, we obtain
 \begin{equation*}
     L^{}_{k+\beta'+1}(\Gamma)=\min_{B\in \mathcal{F}_{k+\beta'+1}(\mathcal{S}_2(V))}\max_{f\in B}\mathcal{R}^{\sigma}_{2,\infty}(f)=\max_{f\in B_0}\mathcal{R}^{\sigma}_{2,\infty}(f)\geq\mathcal{R}^{\sigma}_{2,\infty}(g)\geq C_0.
 \end{equation*}
 Observe that 
\[C_0 \geq \min_{1\leq m\leq l}\frac{w_{\min}|E_m|}{\mu_{\max}|V_m|}=\frac{w_{\min}}{\mu_{\max}}\min_{1\leq m\leq l}\frac{|V_m|-1}{|V_m|}\geq\frac{1}{2}\frac{w_{\min}}{\mu_{\max}}.\]
This concludes the proof.
\end{proof}

Next, we give some remarks about Theorem \ref{thm:Inertia vertices} and Theorem \ref{thm:Inertia edge}. 

For a signed graph $\Gamma=(G,\sigma)$ whose cardinality of a maximum independent set is $\alpha$, we get from Theorem \ref{thm:Inertia vertices} that the $\alpha$-th variational eigenvalue $\lambda_{\alpha}^{(p)}$ of the $p$-Laplacian times $2^{-p}$ tends to zero as $p$ tends to infinity. The next proposition shows how fast it tends to zero. Indeed, the $k$-th variational eigenvalue $\lambda_k^{(p)}$ is uniformly bounded with respect to $p$ whenever $k$ is no greater than $\alpha$.
\begin{pro}\label{pro:alphabdd}
   Let $\Gamma=(G,\sigma)$ be a signed graph where $G=(V,E)$. If the cardinality of a maximal independent set of the graph $G=(V,E)$ is $\alpha$, then $$\lambda_{\alpha}^{(p)}\leq \max_{1\leq i\leq n}\frac{\sum_{j\sim i}w_{ij}+\kappa_i}{\mu_i}.$$
\end{pro}
\begin{proof}
    As in the proof of Theorem \ref{thm:Inertia vertices}, we assume that $\{1,2,\ldots,\alpha\}\subset V$ is a maximal independent set. Similarly,  define $W$ to be the linear space spanned by the functions $g_1,\ldots g_{\alpha}$, where the $g_i(i)=1$ and $g_i(j)=0$ when $j\neq i$ for $1\leq i \leq \alpha$. By Proposition \ref{pro:index} $(i)$, we get $\gamma(W\cap \mathcal{S}_p)=\alpha$ for any $p> 1$. Take any $f\in W\cap \mathcal{S}_p$, we obtain
\begin{equation*} \mathcal{R}^{\sigma}_p(f)=\frac{\sum_{\{i,j\}\in E}w_{ij}|f(i)-\sigma_{ij}f(j)|^p+\sum_{i=1}^n\kappa_i|f(i)|^p }{\sum_{i=1}^{n}\mu_i|f(i)|^p}=\frac{\sum_{i=1}^{\alpha}\left(\sum_{j\sim i}w_{ij}+\kappa_i\right)|f(i)|^p}{\sum_{i=1}^{\alpha}\mu_i\left|f(i)\right|^p}\leq C,
\end{equation*}
where $C:=\max_{1\leq i\leq n}\frac{\sum_{j\sim i}w_{ij}+\kappa_i}{\mu_i}$. By definition, we have
\begin{equation*}
    \lambda_{\alpha}^{(p)}=\min_{B\in \mathcal{F}_{\alpha}(\mathcal{S}_p)}\max_{g\in B}\mathcal{R}^{\sigma}_p(g)\leq\max_{f\in W\cap \mathcal{S}_p}\mathcal{R}^{\sigma}_p(f)\leq C. 
\end{equation*}
The proof is completed.
\end{proof}

For a graph with isolated vertices, there exists no edge covers. We can first remove all isolated vertices and then apply the above Lemma \ref{lemma:edge cover} to derive an upper bound of the number of the zero elements in the sequence $(L_1(\Gamma),L_2(\Gamma)\ldots,L_n(\Gamma))$.  For a graph without any isolated vertices, we derive directly the following corollary.
 \begin{cor}\label{cor:edge cover 2}
     Let $\Gamma=(G,\sigma)$ be a signed graph without isolated vertices. Let $\beta$ be the cardinality of a minimum edge cover of $G$. Then, we have $L_{\beta+1}(\Gamma)>0$.
\end{cor}
On the other hand, for a signed graph $\Gamma$ without isolated vertices, we can still apply Lemma \ref{lemma:edge cover} to derive more upper bounds of the number of zero elements in the sequence $(L_1(\Gamma),L_2(\Gamma)\ldots,L_n(\Gamma))$ by removing certain vertices and using the minimum edge covers of the resulting subgraphs. 
However, the proposition below tells that those new-obtained upper bounds do not improve the one from Corollary \ref{cor:edge cover 2}.
\begin{pro}
  Let $G=(V,E)$ be a graph without isolated vertices and $G'=(V',E')$ be an induced subgraph of $G$ obtained by removing $k$ vertices from $G$. If the cardinality of a minimum edge cover of $G'$ is $m$, then  the cardinality of a minimum edge cover of $G$ is less than or equal to $m+k$.
\end{pro}
\begin{proof}
  We first prove the case when $k=1$. Assume $G'$ is induced by $V\setminus\{x\}$ and a minimum edge cover of $G'$ is $\{e_1,e_2,\ldots,e_m\}$. Since $x$ is not an  isolated vertex, there exists an edge $e\in E$ such that $x$ is incident with $e$. So $\{e_1,e_2,\ldots,e_m,e\}$ is an edge cover of $G$. This proves the case $k=1$.
In general, the proof follows by applying the above arguments recursively. 
\end{proof}

As an application of Lemma \ref{lemma:edge cover}, we have the following asymptotic result about the $p$-Laplacian variational eigenvalues. 

\begin{pro}\label{pro:asymptotic}
    Let $\Gamma=(G,\sigma)$ be a signed graph without isolated vertices, and $\beta$ be the cardinality of a minimum edge cover of $G$. Then we have for any $\beta+1\leq k\leq n$
    \[2^{p-1}\frac{w_{\min}}{\mu_{\max}}-\max_{1\leq i\leq n}\left|\frac{\kappa_i}{\mu_i}\right|\leq \lambda_k^{(p)}\leq 2^{p-1}\max_{1\leq i\leq n}\frac{\sum_{j\sim i}w_{ij}+\kappa_i}{\mu_i}.\]
    That is, $\lambda_k^{(p)}=\Theta(2^{p})$ as $p$ tends to infinity, whenever $k> \beta$.
\end{pro}
\begin{proof}
    The upper bound follows from a direct observation that
   \[\mathcal{R}^\sigma_p(f)\leq 2^{p-1} \max_{1\leq i\leq n}\frac{\sum_{j\sim i}w_{ij}+\kappa_i}{\mu_i}\] 
   for any nonzero function $f\in C(V)$. The lower bound follows from the monotonicity \eqref{eq:de}  in Theorem \ref{thm:Monotonicity}, Proposition \ref{lemma:not depend} and Lemma \ref{lemma:edge cover} as below: For $k>\beta$
   \[2^{-p}\left(\lambda_k^{(p)}+\mathcal{C}\right)\geq L_{k}\geq \frac{1}{2}\frac{w_{\min}}{\mu_{\max}},\]
   where $\mathcal{C}:=\max_{1\leq i\leq n}\left|\frac{\kappa_i}{\mu_i}\right|$.
\end{proof}

Combining Theorem \ref{thm:Inertia vertices}, Proposition \ref{pro:alphabdd}, Corollary \ref{cor:edge cover 2}, and Proposition \ref{pro:asymptotic}, we get the following characterization of the growth rate of $p$-Laplacian variational eigenvalues of bipartite graphs.
\begin{cor}
  If $\Gamma=(G,\sigma)$ is a bipartite graph without isolated vertices, and $\beta$ be the cardinality of a minimum  edge cover of $G$. Then we have $L_\beta(\Gamma)=0$ and $L_{\beta+1}(\Gamma)>0$. Moreover, the $p$-Laplacian variational eigenvalues $\lambda_k^{(p)}$ is bounded for $1\leq k\leq \beta$, and is $\Theta(2^p)$ for $\beta+1\leq k\leq n$, as $p$ tends to infinity.
\end{cor}
\begin{proof}
This is due to the following fact: For a bipartite graph without isolated vertices, the cardinality of a minimum edge cover coincides with the cardinality of a maximum independent set.
\end{proof}

To conclude this section, we prove that our Theorem \ref{thm:Inertia vertices} covers the Cvetkovi$\acute{c}$ bound.

For any $M=(m_{ij})_{1\leq i,j\leq n}\in \mathcal{M}$, there is no restriction of $m_{ij}$ if $i\sim j$.  Next, We prove that if we request that $m_{ij}\neq 0$ whenever $i\sim j$, the upper bound is the same as the Cvetkovi$\acute{c}$ bound.
\begin{lemma}\label{lemma:M_0}
      Let $G=(V,E)$ be 
 a graph and $\mathcal{M}$ be defined as in Theorem \ref{thm:Inertia bound}. Define
     $$\mathcal{M}_0:=\{M: M=(m_{ij})_{1\leq i,j\leq n} \text{ is a symmetric matrix such that $m_{ij}=0$ if and only if $i\nsim j$}\}.$$ Then, we have
      $$
    \min_{M\in \mathcal{M}_0}\min\{n-n^+(M),n-n^-(M)\}=\min_{M\in \mathcal{M}}\min\{n-n^+(M),n-n^-(M)\}.$$
\end{lemma}
\begin{proof}
   Since $\mathcal{M}_0\subset \mathcal{M}$, we have $$
    \min_{M\in \mathcal{M}_0}\min\{n-n^+(M),n-n^-(M)\}\geq \min_{M\in \mathcal{M}}\min\{n-n^+(M),n-n^-(M)\}.$$
It remains to prove the converse. Let $M'=(m'_{ij})_{1\leq i,j\leq n}\in\mathcal{M}$ be a matrix such that  $$
    \min\{n-n^+(M'),n-n^-(M')\}=\min_{M\in \mathcal{M}}\min\{n-n^+(M),n-n^-(M)\}.$$
For any $\epsilon>0$, we define a new matrix $M_{\epsilon}''=(m''_{ij})_{1\leq i,j\leq n}\in \mathcal{M}_0$ as follows
\begin{equation*}
   m''_{ij}=\left\{ \begin{aligned}
        &m'_{ij}\quad &&  i\sim j \text{ and } m'_{ij}\neq 0, \\
        &\epsilon \quad && i\sim j\text{ and }m'_{ij}= 0,\\
        &0  \quad &&i\nsim j.
    \end{aligned}
    \right.
\end{equation*}
By continuity of the eigenvalues, we have $$n^+(M''_{\epsilon})\geq n^+(M')\quad\text{ and }\quad n^-(M''_{\epsilon})\geq n^-(M'),$$
when $\epsilon$ is sufficiently small. Therefore, we estimate
\begin{equation*}
    \begin{aligned}
       \min_{M\in \mathcal{M}_0}\min\{n-n^+(M),n-n^-(M)\}&\leq \min\{n-n^+(M''_{\epsilon}),n-n^-(M''_{\epsilon})\}\\
       &\leq \min\{n-n^+(M'),n-n^-(M')\}\\    
       &=\min_{M\in \mathcal{M}}\min\{n-n^+(M),n-n^-(M)\}.  
    \end{aligned}
\end{equation*}
This completes the proof.
\end{proof}
The following proposition shows that Theorem \ref{thm:Inertia vertices} covers the Cvetkovi\'c bound in Theorem \ref{thm:Inertia bound}.
\begin{pro}\label{pro:cover}
    Let $G=(V,E)$ be a graph. For any edge weight $w$ and vertex measure $\mu$, we have 
 $$\min_\sigma\#\{k:L_k(\Gamma^{\sigma})=0\}\leq  \min_{M\in \mathcal{M}_0}\min\{n-n^+(M),n-n^-(M)\}, $$
 where $\mathcal{M}_0$ is defined as in Lemma \ref{lemma:M_0}.
\end{pro}
\begin{proof}
  Let $M'=(m'_{ij})_{1\leq i,j\leq n}\in\mathcal{M}_0$ be a matrix satisfying
\begin{equation*}
    \begin{aligned}
       n-n^+(M')&=\min\{n-n^+(M'),n-n^-(M')\}=\min_{M\in \mathcal{M}_0}\min\{n-n^+(M),n-n^-(M)\}.  
    \end{aligned}
\end{equation*}
If $n-n^-(M')=\min\{n-n^+(M'),n-n^-(M')\}$, we simply take $-M'$ instead of $M'$.
For simplicity, we  denote $n-n^+(M')$ by $m$. Define an edge weight $w$, a vertex measure $\mu$ and a signature $\sigma$ as follows:
\begin{equation*}
    \begin{aligned}
w_{ij}&:=\left|m'_{ij}\right|,\qquad &&\text{for any } \{i,j\}\in E; \\
\sigma_{ij}&:=-\text{sign}(m'_{ij}),\qquad &&\text{for any } \{i,j\}\in E; \\
     \mu_i&:=1,\qquad &&\text{for any } i\in V.
    \end{aligned}
\end{equation*}
By definition, we have $A_{-\Gamma}=M'$. Moreover, since the vertex measure  $\mu\equiv1$, we have $A^{\mu}_{-\Gamma}=A_{-\Gamma}$. By Theorem \ref{thm:L_n's antibalanced}, we get $$L_{m+1}(\Gamma)\geq \frac{1}{2} \lambda_{m+1}\left(A^{\mu}_{-\Gamma}\right)=\frac{1}{2} \lambda_{m+1}\left(A_{-\Gamma}\right)=\frac{1}{2}\lambda_{m+1}(M)>0.$$
Therefore, we arrive at
\begin{equation*}
    \begin{aligned}
      \min_\sigma\#\{k:L_k(\Gamma^{\sigma})=0\} &\leq  \#\{k:L_k(\Gamma)=0\}\\
      &\leq m =n-n^+(M')\\
      &=\min_{M\in \mathcal{M}_0}\min\{n-n^+(M),n-n^-(M)\}.
    \end{aligned}
\end{equation*}
The desired conclusion now follows from Proposition \ref{pro:weight and measure}.
\end{proof}

\section{Tensor eigenvalues}\label{Tensor}
In this section, we establish a connection between graph $p$-Laplacian eigenvalues and  tensor eigenvalues.


We first recall the concepts of tensors and their eigenvalues. An $m$-th order $n$-dimensional tensor $\mathcal{A}$ consists of $n^m$  entries $A_{i_1,\ldots,i_m}\in \mathbb{R}$, where $i_l\in\{1,\ldots,n\}$ for $l=1,\ldots,m$. A tensor $\mathcal{A}$ is called symmetric if its entries are invariant under any permutation of their indices. 

Given an $m$-th order $n$-dimensional tensor $\mathcal{A}$, we define an operator $\mathcal{A}:\mathbb{R}^n\to\mathbb{R}^n $ as follows: 
$$\left(\mathcal{A}f^{m-1}\right)(i)=\sum_{i_2,\ldots,i_m=1}^n\mathcal{A}_{i,i_2,\ldots,i_m}f(i_2)\cdots f(i_m),\qquad\,\text{for any}\,\, f\in \mathbb{R}^n \,\,\text{and}\,\,1\leq i\leq n,$$
where we use $\mathcal{A}f^{m-1}$ to denote $\mathcal{A}(f)$.
\begin{defn}[\cite{Lim15,Qi05}]
    Let $\mathcal{A}$ be an $m$-th order $n$-dimensional tensor. We say that a non-zero vector $f$ is an eigenvector of $\mathcal{A}$ if there exists a constant $\lambda\in \mathbb{R}$ such that \[\left(\mathcal{A}f^{m-1}\right)(i)=\lambda f(i)^{m-1}\qquad 1\leq i\leq n. \]
We call $\lambda$  an eigenvalue of the tensor $\mathcal{A}$.
\end{defn}
Let $\Gamma=(G,\sigma)$ be a signed graph with $G=(V,E)$. Recall that its graph $p$-Laplacian $\Delta_p^{\sigma}$ is defined as follows
$$\Delta_p^{\sigma}f(i)=\sum_{j\sim i}w_{ij}\Psi_{p}\left(f(i)-\sigma_{ij}f(j)\right)+\kappa_i\Psi_p\left(f(i)\right),\qquad \forall 1\leq i\leq n.$$

We have the following interesting observation. 
\begin{pro}\label{prop:tensor}
   Let $\Gamma=(G,\sigma)$ be a signed graph with $n$ vertices. For any positive even number $p$, there exists a $p$-th order $n$-dimensional symmetric tensor $\mathcal{A}^{(p)}$ such that
   \[\Delta_p^\sigma f=\mathcal{A}^{(p)}(f)=:\mathcal{A}^{(p)}f^{p-1},\,\,\text{for any}\,\,f\in C(V).\]
\end{pro}
\begin{proof}
Note that any $f\in C(V)$ can be regarded as a vector in $\mathbb{R}^{n}$. When $p$ is an positive even number, we can define a $p$-th order $n$-dimensional symmetric tensor $\mathcal{A}^{(p)}$ such that $\Delta_p^{\sigma}f=\mathcal{A}^{(p)}f^{p-1}$ for any $f:V \to \mathbb{R}$ as follows
\begin{equation}\label{eq:tensor}
   \mathcal{A}^{(p)}_{i,i_2,\ldots,i_{p}}=\left\{ \begin{aligned}
        &0 && \qquad |\{i,i_2,\ldots,i_{p}\}|\geq 3,\\
        &0 && \qquad \{i,i_2,\ldots,i_{p}\}=\{i,j\}\,\,\text{and}\,\,i \nsim j,\\
        &(-\sigma_{ij})^lw_{ij} && \qquad \{i,i_2,\ldots,i_{p}\}=\{i^{(l)},j^{(p-l)}\}\,\,\text{and}\,\,i \sim j\,\,,\\
        &\kappa_i+\sum_{j\sim i}w_{ij}  &&  \qquad \{i,i_2,\ldots,i_{p}\}=\{i\}.
    \end{aligned}
    \right.
\end{equation}
where $|\{i,i_2,\ldots,i_{p}\}|$ is the cardinality of set $\{i,i_2,\ldots,i_{p}\}$ and $\{i,i_2,\ldots,i_{p-1}\}=\{i^{(l)},j^{(p-l)}\}$ means there are $l$ $i$'s and  $(p-l)$ $j$'s in  $\{i,i_2,\ldots,i_{p}\}$.
By a direct computation, for any $f:V\to \mathbb{R}$ and $i\in V$, we have 
\begin{equation*}
\begin{aligned}
    \mathcal{A}^{(p)}f^{p-1}(i)&=\sum_{i_2,\ldots,i_p=1}^n\mathcal{A}^{(p)}_{i,i_2,\ldots,i_p}f(i_2)\cdots f(i_p)\\
    &=\left(\kappa_i+\sum_{j\sim i}w_{ij}\right)\left(f(i)\right)^{p-1}+\sum_{j\sim i}\sum_{l=1}^{p-1}\tbinom{p-1}{l}(-\sigma_{ij})^lw_{ij}\left(f(i)\right)^{p-1-l}\left(f(j)\right)^{l}\\
    &=\sum_{j\sim i}w_{ij}\left(f(i)-\sigma_{ij}f(j) \right)^{p-1}+\kappa_i\left(f(i)\right)^{p-1}\\
    &=\Delta_{p}^{\sigma}f(i).
\end{aligned}
\end{equation*}
This equality shows that the operator $\Delta_p^{\sigma}$ coincides with the operator $\mathcal{A}^{(p)}$. 
\end{proof}

Due to the above proposition, the spectrum of graph $p$-Laplacian for $p$ even coincides with that of a tensor in the following sense: For any eigenpair $(\lambda, f)$ of $\Delta_p^\sigma$ with $p$ even, i.e., $\Delta_p^\sigma f(i)=\lambda \mu_i f(i)^{p-1}$, we have 
\[\mathcal{A}f^{p-1}=\lambda f^{p-1},\]
where 
\begin{equation}\label{eq:tensorNormalized}
\mathcal{A}_{i_1,\ldots,i_p}=\frac{1}{\mu_{i_1}}\mathcal{A}^{(p)}_{i_1,\ldots,i_p},\,\,\text{for any}\,\,i_1,\ldots, i_p,
\end{equation}
with $\mathcal{A}^{(p)}$ defined in (\ref{eq:tensor}). Notice that, the tensor $\mathcal{A}$ here is not necessarily symmetric. 

Combining the above observation with Theorem \ref{thm:anti-eigen}  and Theorem \ref{thm:lower bound}, we get the following theorem.

\begin{theorem}\label{thm:tensor}
Let $\Gamma=(G,\sigma)$ be a signed graph, $\mathcal{A}$ be the tensor defined as in Eq.~\eqref{eq:tensorNormalized} and $\eta^{(p)}$ be the largest eigenvalue of $\mathcal{A}$. Denote by $\mathcal{G}$ the set consisting of all subgraphs $\Gamma_0$ of $\Gamma$ and $|V_{\Gamma_0}|$  the number of vertices of $\Gamma_0$. Define $\mathcal{C}:=\max_{1\leq i\leq n}\left|\frac{\kappa_i}{\mu_i}\right|$. We have   $$\eta^{(p)}\ge2^{p-1}\max_{\Gamma_0\in \mathcal{G}}\lambda_{|V_{\Gamma_0}|}(A^{\mu}_{-\Gamma_0})-\mathcal{C}=2^{p-1}\max_{\substack{\Gamma_0\in \mathcal{G}\\|V_{\Gamma_0}|=n } }\lambda_{n}(A^{\mu}_{-\Gamma_0})-\mathcal{C}.$$ 
  Moreover, when $\sigma\equiv -1$, define function $f:V\to \mathbb{R}$ as follows:
        $$f(i):=\lim_{p\to \infty}\left|f^{(p)}(i)\right|^{\frac{p}{2}},\qquad\forall i\in V,$$ 
    where $f^{(p)}\in \mathcal{S}_p$ is an eigenfunction of $\mathcal{A}$ corresponding to $\eta^{(p)}$.
     Then $f$  is well defined and it is an eigenfunction of 
       $A^{\mu}_{-\Gamma}$ corresponding to $\lambda_n(A^{\mu}_{-\Gamma})$. 
\end{theorem}

\section{Examples}
\label{sec:exam}
In this section, we discuss some examples whose cut-off adjacency eigenvalues can be explicitly calculated. Moreover, we present examples of graphs which can be completely determined by their cut-off adjacency eigenvalues. 

We assume $w\equiv1$, $\mu\equiv1$ and $\kappa\equiv0$ throughout this section. Recall that we regard an unsigned graph $G=(V,E)$ as a signed graph with the signature $\sigma\equiv+1$. We define $\mathrm{deg}(i)=\sum_{j\sim i}1$. We first characterize the $L_n$ of connected regular bipartite graphs. For this purpose, we recall the following lemma.
\begin{lemma}[{\cite[Lemma 2.1]{Amghibech06}}]\label{lemma:upper}
    Let $G$ be a connected bipartite graph. Then
    $$\lambda_{n}^{(p)}\geq 2^{p-1}\min\{\mathrm{deg}(i),i\in V\},$$
    with equality if and only if $G$ is a regular bipartite graph.
\end{lemma}

\begin{example}\label{exam:bipartite}
   Let $G=(V,E)$ be a connected regular bipartite graph. We have $$L_n:=\lim_{p\to \infty}2^{-p}\lambda^{(p)}_n=\frac{d}{2},$$ where $d=\mathrm{deg}(i)$ for any $i\in V$. 
\end{example} 
\begin{proof}
     Since $G$ is a connected regular bipartite graph, we have $\lambda_n^{(p)}= 2^{p-1}d$  by Lemma \ref{lemma:upper}. By definition of $L_n$, we have $L_n=\lim_{p\to \infty}2^{-p}\lambda^{(p)}_n=\frac{d}{2}$. 
\end{proof}

One of the most important topics in spectral graph theory is to categorize graphs via their eigenvalues. In this spirit, we explore examples of graphs that can be completely determined by their cut-off adjacency eigenvalues $(L_1,L_2,\ldots,L_n)$. 
\begin{theorem}\label{thm:noedge}
    $(L_1,L_2,\ldots,L_n)=(0,0,\ldots,0)$ if and only if the signed graph $\Gamma=(G,\sigma)$ with $G=(V,E)$ has no edges.
\end{theorem}
\begin{proof}
    Sufficiency: If  $\Gamma$ has no edges, we have $\mathcal{R}^{\sigma}_p(g)=0$ for any function $g:V\to \mathbb{R}$ and any $p$.  By definition, we derive
    \begin{equation*}
    \lambda_{k}^{(p)}=\min_{B\in \mathcal{F}_{k}(\mathcal{S}_p)}\max_{g\in B}\mathcal{R}^{\sigma}_p(g)=0,
\end{equation*}
and, hence $L_k=\lim_{p\to\infty}2^{-p}\lambda_k^{(p)}=0$,
 for any $1\leq k\leq n$. 

 Necessity: We argue by contradiction. Suppose there is an edge $e=\{1,2\}\in E$. We  remove all vertices but $\{1,2\}$ to get an induced subgraph $\Gamma'=(G',\sigma)$ with $G'=(V',E')$, where $V'=\{1,2\}$ and $E'=\left\{\{1,2\}\right\}$. By definition, the cardinality of the minimum  edge cover of graph $G'$ is $1$. Using Lemma \ref{lemma:edge cover}, we get $L_{n}>0$. Contradiction.
\end{proof}
Let us recall the following result due to Amghibech \cite[Theorem 6]{Amghibech}. 
\begin{theorem}[{\cite[Theorem 6]{Amghibech}}]\label{thm:complete}
   Let $G=(V,E)$ be a complete graph. Assume $n\geq 3$. Let $q$ be  the H\"older conjugate to $p$, i.e., $\frac{1}{p}+\frac{1}{q}=1$. The positive eigenvalues of $p$-Laplacian of $G$ are $$ n-(h+k)+(h^{q-1}+k^{q-1})^{p-1},$$ where $h,k$ are positive integers with $h+k\leq n$. 
\end{theorem}

\begin{theorem}\label{thm:L-compete}
Let $G=(V,E)$ be a graph. Then $L_2>0$ if and only if  $G$ is complete. Moreover, when $G$ is complete, we have $L_2=\frac{1}{2}$ and \begin{equation*}
 L_n=\left\{ \begin{aligned}
        &\frac{n}{4}& \qquad n\text{ is even},\\
 &\frac{\sqrt{n^2-1}}{4}  & \qquad n\text{ is odd}.
    \end{aligned}
    \right.
\end{equation*} 
\end{theorem}
\begin{proof}

We first assume $|V|=2$.

Sufficiency: Let $G$ be a complete graph. Then $G$ is a regular bipartite graph. By Example \ref{exam:bipartite}, we have $L_1=0$ and $L_2=\frac{1}{2}$.

Necessity: If $G$ is not a complete graph, then there is no edges in $G$. We use Theorem \ref{thm:noedge} to get $L_2=0$. Contradiction.

Next, we  assume $|V|\geq3$.

     Sufficiency: Assume graph $G$ is a complete graph. By \cite[Theorem 4.1]{DPT21}, we have $\lambda_1^{(p)}=0$ is simple. Since there is no eigenvalue between $\lambda_1^{(p)}$ and $\lambda_2^{(p)}$, we get by Theorem \ref{thm:complete} $$\lambda_2^{(p)}=\min_{\substack{h+k\leq n\\  h,k\in \{ 1,2,\ldots,n-1\} }      }\big( n-(h+k)+(h^{q-1}+k^{q-1})^{p-1} \big).$$
Next, we compute 
$$L_2=\lim_{p\to \infty}2^{-p}\lambda_2^{(p)}=\min_{\substack{h+k\leq n\\  h,k\in \{ 1,2,\ldots,n-1\}   }      }\lim_{p\to\infty}2^{-p}(h^{q-1}+k^{q-1})^{p-1}=\frac{1}{2}. $$
Similarly, we use the fact that $\lambda_n^{(p)}$ is the largest eigenvalue of the $p$-Laplacian. By Theorem \ref{thm:complete}, we get $$\lambda_n^{(p)}=\max_{\substack{h+k\leq n\\  h,k\in \{ 1,2,\ldots,n-1\} }      }\big( n-(h+k)+(h^{q-1}+k^{q-1})^{p-1} \big).$$ Then by computation, we have $$L_n=\lim_{p\to \infty}2^{-p}\lambda_n^{(p)}=\max_{\substack{h+k\leq n\\  h,k\in \{ 1,2,\ldots,n-1\} }      }\lim_{p\to\infty}2^{-p} (h^{q-1}+k^{q-1})^{p-1}=\max_{\substack{h+k\leq n\\  h,k\in \{ 1,2,\ldots,n-1\}    }      }\frac{\sqrt{hk}}{2}. $$
So when $n$ is even, we have $L_n=\frac{n}{4}$. When $n$ is odd, we have $L_n=\frac{\sqrt{n^2-1}}{4}$. 

Necessity: Since $L_2>0$, we have  by Theorem \ref{thm:Inertia vertices} that the cardinality of a maximum independent set of $G$ is $1$. So $G$ must be a complete graph.
\end{proof}

We now turn to a complete characterization of graphs with $L_{n-1}=0$, which is given in Theorem \ref{thm:star}.  For this purpose, we give some preparatory work. First, we give a characterization of a graph on $n$ vertices whose minimum edge cover has cardinality $n-1$. 
\begin{pro}\label{pro:star}
    If the cardinality of a minimum  edge cover of a graph $G=(V,E)$ is $n-1$, then $G$ is the complete graph $K_3$ with $3$ vertices or a star graph.
\end{pro}
This is well known to experts. For the readers' convenience, we provide a proof below.
\begin{proof}
Since $G$ has edge covers, it has no isolated vertices. 

When $n=2$, if the cardinality of a minimum edge cover is $1$, then $G$ must be $K_{2}$.

 When $n=3$, if the cardinality of a minimum edge cover  is $2$, then $G$ must be either $K_3$ or $K_{2,1}$ (i.e., the star graph with $3$ vertices). 

 We next consider the case $n\geq 4$. By Gallai' theorem (see, e.g., \cite[Theorem 3.1.22]{West}), the cardinality of a maximum matching of $G$ is $1$. Without loss of generality, we assume that $e=\{1,2\}$ is a maximum matching of $G$. 
 By definition, for any edge in $E$, one of its endpoints must be the vertex $1$ or  the vertex 
 $2$. Otherwise, there is at least one edge $e_0\in E$ which is independent to $e$. Then $\{e,e_0\}$ is a matching of size $2$. Contradiction. 
Since $G$ has no isolated vertices, any vertex $i$ different from $1$ and $2$ must be adjacent to $1$ or $2$.
 We consider the following two cases.

 Case $1$: $\mathrm{deg}(1)=1$ or $\mathrm{deg}(2)=1$. If $\mathrm{deg}(1)=1$, then the vertex $2$ is incident to every edge in $E$. So $G$ is star graph. The same proof is effective if $\mathrm{deg}(2)=1$.

  Case $2$: $\mathrm{deg}(1)\geq 2$ and $\mathrm{deg}(2)\geq 2$. Without loss of generality, we can assume $3\sim 1$. If there exists $i\in V$ where $i\neq 1,2,3$ such that $i\sim 2$. Then $\left\{ \{1,3\},\{2,i\} \right\}$ is a matching of size $2$. Contradiction. If the neighbourhood of $2$ is $\{1,3\}$, since $n\geq 4$, we  take a vertices $n\in V$ such that $n\sim 1$. Then $\left\{ \{1,n\},\{2,3\} \right\}$ is a matching of size $2$. Contradiction.
  
So we conclude that $G$ must be a star graph when $n\geq 4$. 
\end{proof}
When the graph $G=(V,E)$ is disconnected, we denote its connected components by $\{G_i\}_{i=1}^l$. While the set of $p$-Laplacian eigenvalues of $G$ is the union of $p$-Laplacian eigenvalues of each $G_i$, it is not known whether the set of $p$-Laplacian \emph{variational} eigenvalues of $G$ is the union of $p$-Laplacian \emph{variational} eigenvalues of each $G_i$ or not. The latter is only known to be true for forests \cite{DPT21}. In general, we have the following result.

\begin{pro}\label{pro:union}
 Let $G=(V,E)$ be a disconnected graph with $l$ connected components which are denoted by $\{G_i\}_{i=1}^l$ with $G_i=(V_i,E_i)$. Let $\Delta_p(G)$ and $\Delta_p(G_i)$ be the $p$-Laplacians of $G$ and $G_i$, respectively. For any $p>1$, we have $$\lambda_1^{(p)}(G)=\lambda_2^{(p)}(G)=\cdots\lambda_l^{(p)}(G)=0$$ and $$\lambda_{l+1}^{(p)}(G)=\min_{1\leq i\leq l}\lambda_{2}^{(p)}(G_i),$$
 where $\lambda_k^{(p)}(G)$ and $\lambda_k^{(p)}(G_i)$ are the $k$-th variational eigenvalues of $\Delta_p(G)$ and  $\Delta_p(G_i)$, respectively.
\end{pro}
\begin{proof}
    For each $j\in \{1,2,\ldots,l\}$,
we define a function $f_j$ as below
\begin{equation}\label{eq:f_jconnectedcomponent}
f_j(x)=\left\{
	\begin{aligned}
		&1 &\qquad x\in V_j, \\
		&0 &\qquad \mathrm{otherwise.}
	\end{aligned}
	\right.
\end{equation}
 Let $W$  be the linear space spanned by the functions $\{f_j\}_{j=1}^l$. By Proposition \ref{pro:index} $(i)$, we have $\gamma\left(W\cap \mathcal{S}_p(V)\right)=l$. By definition of variational eigenvalues, we derive
 \begin{equation*}
   \lambda_l^{(p)}(G)=\min_{B\in \mathcal{F}_l\left(\mathcal{S}_p(V)\right)}\max_{f\in B}\mathcal{R}_p^{G}(f)\leq \max_{f\in W\cap \mathcal{S}_p}\mathcal{R}_p^{G}(f)=0, 
 \end{equation*}
where $\mathcal{R}_p^{G}$ is the Rayleigh quotient of $\Delta_p(G)$. Since $\kappa\equiv0$, we have $\lambda_l^{(p)}(G)\geq \lambda_1^{(p)}(G)\geq 0$. This concludes the proof of the first equality.
 
 For the second equality, we assume without loss of generality that $\lambda_{2}^{(p)}(G_1)=\min_{1\leq i\leq l}\lambda_{2}^{(p)}(G_i)$. Let $g_0:V_1\to \mathbb{R}$ be an eigenfunction corresponding to $\lambda_2^{(p)}(G_1)$. Let $V_1^+$ and $V_1^-$ be the sets $\{x\in V_1:g_0(x)>0\}$ and $\{x\in V_1:g_0(x)<0\}$, respectively. Define $g_1:V\to \mathbb{R}$ and $g_2:V\to \mathbb{R}$ as follows
 \begin{equation*}
g_1(x)=\left\{
	\begin{aligned}
		&g_0(x) &\qquad x\in V_1^+,   \\
		&0 &\qquad \mathrm{otherwise,}
	\end{aligned}
	\right.\,\,\,\,\text{and}\,\,\,\,
g_2(x)=\left\{
	\begin{aligned}
		&g_0(x) &\qquad x\in V_1^-,  \\
		&0 &\qquad \mathrm{otherwise.}
	\end{aligned}
	\right.
\end{equation*}
We use \cite[Theorem 4.1]{DPT21} to get that both $g_1$ and $g_2$ are non-zero. Let $W_1$ denote the linear space spanned by $\{f_2,f_3,\ldots,f_l,g_1,g_2\}$ where the functions $f_j$, $2\leq j \leq l$, are defined in (\ref{eq:f_jconnectedcomponent}). By Proposition \ref{pro:index} $(i)$, we have $\gamma\left(W_1\cap \mathcal{S}_p(V)\right)=l+1$. For any function $f$ in $W_1\cap \mathcal{S}_p(V)$, we denote it by  $f=\sum_{i=2}^l k_if_i+c_1g_1+c_2g_2$ where $k_2,\ldots, k_l, c_1,c_2\in \mathbb{R}$. If $c_1=c_2=0$, then we have $\mathcal{R}_p^G(f)=0$. If either $c_1$ or $c_2$ is nonzero, we compute
  \begin{equation*}
	\begin{aligned}
		 \mathcal{R}_p^G(f)&=\frac{\sum_{\{i,j\}\in E}|f(i)-f(j)|^p}{\sum_{i=1}^n|f(i)|^p}   \\	
        & =\frac{\sum_{\{i,j\}\in E_1}\left|c_1g_1(i)+c_2g_2(i)-c_1g_1(j)-c_2g_2(j)\right|^p}{\sum_{i\in V_1}|c_1g_1(i)+c_2g_2(i)|^p+\sum_{k=2}^l\sum_{j\in V_k}|f(j)|^p}\\
        &\leq \frac{\sum_{\{i,j\}\in E_1}\left|c_1g_1(i)+c_2g_2(i)-c_1g_1(j)-c_2g_2(j)\right|^p}{\sum_{i\in V_1}|c_1g_1(i)+c_2g_2(i)|^p} \\
        &\leq \lambda_2^{(p)}(G_1),
	\end{aligned}
\end{equation*}
 The last equality above is due to \cite[Lemma 3.8]{TudiscoHein18}, observing that $V_1^+$ and $V_1^-$ are the two strong nodal domains of the eigenfunction $g_0$ and $g_1, g_2$ are the restrictions of $g_0$ to $V_1^+, V_1^-$, respectively.
 By definition, we obtain
 \begin{equation}\label{eq:union}
		 \lambda_{l+1}^{(p)}(G)=\min_{B\in \mathcal{F}_{l+1}\left(\mathcal{S}_p(V)\right)}\max_{f\in B}\mathcal{R}_p^{G}(f)\leq \max_{f\in W_1\cap \mathcal{S}_p(V)}\mathcal{R}_p^{G}(f)\leq  \lambda_2^{(p)}(G_1).
\end{equation}
It is direct to see that the set of eigenvalues of $\Delta_p(G)$ is equal to the union of the sets of eigenvalues of $\Delta_p(G_i)$, $1\leq i\leq l$. Moreover, by \cite[Theorem 4.1]{DPT21}, there is no non-variational eigenvalues between the first and the second variational eigenvalues for any connected graph. Then, we derive by (\ref{eq:union}) that 
$\lambda_{l+1}^{(p)}(G)$ must be either $\lambda_{2}^{(p)}(G_1)$ or zero. 

Suppose that  $\lambda_{l+1}^{(p)}(G)$ is zero. Let $B_0\in \mathcal{F}_{l+1}\left(\mathcal{S}_p(V)\right)$ be a minimizing set of $\lambda_{l+1}^{(p)}(G)$ whose existence is due to Lemma \ref{app:2}. Since $\lambda_{l+1}^{(p)}(G)=0$, we have  $\mathcal{R}^G_p(f)=0$ for any $f\in B_0$. This implies that $f$ is constant on every connected components of $G$. So $f\in W$, and hence $B_0\subset W$. By monotonicity of the index, we get $\gamma(B_0)\leq \gamma(W)=l$. Contradiction. Therefore, $\lambda_{l+1}^{(p)}(G)$ must be $\lambda_2^{(p)}(G_1)$. This concludes the proof.
\end{proof}

\begin{pro}\label{pro:union L_n}
    Let $\Gamma=(G,\sigma)$ be a signed graph with $l$ connected components $\Gamma_k=(G_k,\sigma)$. Then, we have $$L_n(\Gamma)=\max_{1\leq k\leq l}L_n(\Gamma_k).$$
\end{pro}
\begin{proof}
    Let us denote $G=(V,E)$ and $G_k=(V_k,E_k)$. This proposition follows from
    $$\lambda_n^{(p)}(\Gamma)=\max_{g\in \mathcal{S}_p(V)}\mathcal{R}^{\Gamma}_p(g)=\max_{1\leq k\leq l}\max_{g\in \mathcal{S}_p(V_k)}\mathcal{R}^{\Gamma_k}_p(g)=\max_{1\leq k\leq l}\lambda_n^{(p)}(\Gamma_k),$$ where  $\mathcal{R}_p^{\Gamma}$ and $\mathcal{R}_p^{\Gamma_k}$ are the Rayleigh quotients of the $p$-Laplacians of $\Gamma$ and $\Gamma_k$, respectively.
\end{proof}

We are now prepared to present the following characterization of star graphs with isolated vertices via their cut-off adjacency eigenvalues.

\begin{theorem}\label{thm:star}
   Let $G=(V,E)$ be a graph. Then $(L_1,L_2,\ldots,L_{n-1},L_n)=(0,0,\ldots,0,C)$ with $C>0$ if and only if the graph  $G$ is a star graph of $4C^2+1$ vertices together with $n-4C^2-1$ isolated vertices.
\end{theorem}
\begin{proof}
 Sufficiency: The cardinality of a maximum independent set of a star graph with isolated vertices is $n-1$. By Theorem \ref{thm:Inertia vertices}, we have $L_1=L_2=\cdots=L_{n-1}=0$. By Proposition \ref{pro:union L_n}, we get for any $p$ that $\lambda_n^{(p)}$ is equal to the largest eigenvalue of the $p$-Laplacian of a star graph. By \cite[Lemma 3.1]{Amghibech06}, the largest eigenvalue of the star graph with $4C^2+1$ vertices is $\left(1+(2C)^{\frac{2}{p-1}}\right)^{p-1}$. Therefore, we compute directly
    \begin{equation*}
        L_n=\lim_{p\to \infty}2^{-p}\left(1+(2C)^{\frac{2}{p-1}}\right)^{p-1}=C.
    \end{equation*}
This proves the  sufficiency.

    Necessity: Assume that $G$ has $k$ isolated vertices. Removing all these isolated vertices from $G$ yields an induced subgraph $G'=(V',E')$. If $G$ does not have isolated vertices, then $k$ is set to be zero. We assume that $G'$ has $l$ connected components and denote them by $\{G_i\}_{i=1}^l$ where $G_i=(V_i,E_i)$. For each $i$, we fix a spanning tree of $G_i$. So we can cover all vertices in $V_i$ by the edges of this spanning tree. We  do this for every connected components to get an edge cover of $G'$ with the cardinality of $\sum_{i=1}^l\left(|V_i|-1\right)$. Applying Lemma \ref{lemma:edge cover} yields that $$L_{n-l+1}=L_{\sum_{i=1}^l\left(|V_i|-1\right)+k+1}>0.$$
    Since $L_{n-1}=0$, we get  $l\leq1$. If $l=0$, then $G$ only has isolated vertices. Contradiction. So we have $l=1$. Denote this only connected component by $V_1$. Moreover, we conclude that the cardinality of a minimum edge cover of $G_1 $ is $n-k-1$. By Proposition \ref{pro:star}, $G_1$ is either $K_3$ or a star graph. Next, we prove that $K_3$ is impossible. If  $G_1$ is $K_3$, we have by Theorem \ref{thm:L-compete}  and Proposition \ref{pro:union} that $L_{n-1}=\lim_{p\to \infty}2^{-p}\lambda^{(p)}_{n-1}= \frac{1}{2}>0$. Contradiction. So  $G_1$ must be a star graph.

    At last, by the same computation as in the proof of sufficiency, we get that the number of vertices of the star graph is $4C^2+1$. 
\end{proof}
To conclude this section, we point out that graphs with no more than 3 vertices are completely determined by their cut-off adjacency eigenvalues. This is a straightforward consequence of Theorems \ref{thm:noedge}, \ref{thm:L-compete}, and \ref{thm:star}.

\begin{example}
Let $G=(V,E)$ be a graph with $|V|=2.$ By Theorem \ref{thm:noedge} and Theorem \ref{thm:star}, we have $(L_1,L_2)=(0,0)$ if and only if $G$ has no edge, and $(L_1,L_2)=(0,\frac{1}{2})$  if and only if $G$ is $K_2$.
\end{example}
\begin{example}
Let $G=(V,E)$ be a graph with $|V|=3.$ Combining Theorem \ref{thm:noedge}, Theorem \ref{thm:L-compete} and Theorem \ref{thm:star}, we have $(L_1,L_2,L_3)=(0,0,0)$ if and only if $G$ has no edge; $(L_1,L_2,L_3)=(0,0,\frac{1}{2})$  if and only if $G$ has only one edge;  $(L_1,L_2,L_3)=(0,0,\frac{\sqrt{2}}{2})$ if and only if $G$ is $K_{2,1}$; $(L_1,L_2,L_3)=(0,\frac{1}{2},\frac{\sqrt{2}}{2})$ if and only if $G$ is $K_3$.
\end{example}

\section{Conclusion and Outlook}\label{sec:outlook}

In this paper, we introduce the cut-off adjacency eigenvalues $\{L_k\}$, defined as the limits of the variational eigenvalues of the $p$-Laplacian, namely, $L_k:=\lim_{p\to \infty}2^{-p}\lambda_k^{(p)}$. Having established several results concerning these eigenvalues, we conclude by briefly discussing their associated eigenfunctions. 


It should be noted that a different limit, i.e., $\lim_{p\to \infty}(\lambda^{(p)}_k)^{1/p}$, is commonly used in the literature to define the $\infty$-Laplacian variational eigenvalues \cite{DBPT26,JL05}. In the graph setting, Deidda, Bruger, Putti, and Tudisco \cite{DBPT26} studied the limiting behavior of eigenfunctions of the $p$-Laplacian as $p\to\infty$. More precisely, let $\left(\lambda_k^{(p_n)},f^{(p_n)}_k\right)$ be a sequence of eigenpairs of graph $p_n$-Laplacian with $p_n\to\infty$ as $n$ goes to $\infty$. If
$$
\left(\lambda_k^{(p_n)}\right)^{1/p_n}\to\lambda
\quad\text{and}\quad
f^{(p_n)}_k\to f \quad \text{pointwise},\quad \text{ as }n\to\infty,
$$
then $f$ satisfies a certain nonlinear equation involving $\lambda$, see Theorem 3.1 in \cite{DBPT26}. Moreover, every nonconstant function satisfying this nonlinear equation is a generalized critical point of the $\infty$-Rayleigh quotient on $G$, see Corollary 4.9 in \cite{DBPT26}. Inspired by the work of Deidda et al., it is therefore natural to investigate the limiting behavior of eigenfunctions in our setting. However, as Theorem \ref{thm:anti-eigen} suggests, in our setting it is more appropriate to consider $\phi_{p,2}(f_k^{(p)})$ rather than $f_k^{(p)}$ itself. This leads us to the following two questions regarding the limiting behaviors in our setting.
\begin{Question}\label{ques:eigen}
   Let $p_n$ be a sequence such that $\lim_{n\to \infty}p_n=+\infty$. If  $\lim_{n\to \infty }\phi_{p_n,2}(f_k^{(p_n)})$ exists, does this limiting function satisfy a nonlinear equation involving $L_k$?
\end{Question}
\begin{Question}
   If such a nonlinear equation exists, does every nontrivial solution of this equation correspond to a generalized critical point of $\mathcal{R}_{2,\infty}^{\sigma}$?
\end{Question}

\label{Application}

\section*{Conflicts of Interest} None.

\section*{Financial Support}
 CG and SL are supported by the National Key R and D Program of China 2020YFA0713100, the National Natural Science Foundation of China (No. 12031017), and Innovation Program for Quantum Science and Technology 2021ZD0302902. DZ is supported by the National Natural Science Foundation of China (No. 12401443). 

\section*{Acknowledgement}
We thank the anonymous referee for many comments and suggestions which greatly helped us improve the quality of the presentation of our paper.
The main structure of this paper was largely determined during DZ's visit to USTC, and DZ is particularly grateful to USTC for its hospitality.

\end{document}